\documentclass[11pt,amssymb]{amsart}
\usepackage{amssymb}
\usepackage{amsmath}
\usepackage[mathscr]{eucal}

\newcommand{\Z}{{\mathbb Z}}

\newcommand{\C}{{\mathbb C}}
\newcommand{\Q}{{\mathbb Q}}

\newcommand{\R}{{\mathbb R}}

\newcommand{\Br}{\mathrm{Br}}

\newcommand{\Tr}{\mathrm{Tr}}

\newcommand{\cO}{\mathscr{O}}

\newcommand{\Ga}{\mathrm{Gal}}
\newtheorem{thm}{Theorem}[section]
\newtheorem{lemma}[thm]{Lemma}
\newtheorem{prop}[thm]{Proposition}
\newtheorem{cor}[thm]{Corollary}

\newcommand{\cG}{\mathcal{G}}

\newcommand{\gen}{\mathbf{gen}}
\newcommand{\Hom}{\mathrm{Hom}}

\newcommand{\uG}{\underline{G}}
\usepackage[all,cmtip]{xy}
\font\brus=wncyr10.240pk scaled 1200 .240pk
\begin{document}

\title[Division algebras with same maximal subfields]{Division algebras with the same maximal subfields}

\author[V.~Chernousov]{Vladimir I. Chernousov}
\author[A.~Rapinchuk]{Andrei S. Rapinchuk}
\author[I.~Rapinchuk]{Igor A. Rapinchuk}

\address{Department of Mathematics, University of Alberta, Edmonton, Alberta T6G 2G1, Canada}

\email{vladimir@ualberta.ca}

\address{Department of Mathematics, University of Virginia,
Charlottesville, VA 22904-4137, USA}

\email{asr3x@virginia.edu}

\address{Department of Mathematics, Harvard University, Cambridge, MA 02138, USA}

\email{rapinch@math.harvard.edu}

\begin{abstract}
We give a survey of recent results related to the problem of characterizing finite-dimensional
division algebras by the set of isomorphism classes of their maximal subfields. We also discuss various
generalizations of this problem and some of its applications. In the last section, we extend the problem
to the context of absolutely almost simple algebraic groups.
\end{abstract}

\maketitle

\hfill {\it To V.P.~Platonov on the occasion of his 75th birthday}

\section{Introduction}\label{S:Introduction}

Our goal in this paper is to give an overview of some recent work on the problem of characterizing a division algebra in terms of its maximal subfields (and, more generally, a simple algebraic group in terms of its maximal tori). Our main focus will be on the following question, as well as some of its variations that will be introduced later on:
%In this survey, we give an expanded overview of the results discussed in the second author's two talks at the Thematic Program on Torsors, Nonassociative Algebras, and Cohomological Invariants at the Fields Institute in Toronto. Just as in the talks, our main focus will be on the following question (as well as some of its variations that will be introduced later on):

%In addition to discussing the topics covered in the lectures, we also include some new results, most notably on the {\it asymmetric genus} of a division algebra.
%introduce in this article the notion of the {\it asymmetric genus} and analyze some examples.

%We will focus primarily on the following question (as well as some of its variations that will be introduced later on):

\vskip2mm

$(\dag)$ \parbox[t]{15.5cm}{{\it What can one say about two finite-dimensional central division algebras $D_1$ and $D_2$ over a field $K$ given that $D_1$ and $D_2$ have the same (isomorphism classes of) maximal subfields?}}

\vskip2mm

\noindent To be precise, the condition on $D_1$ and $D_2$ means that they have the \emph{same degree} $n$ over $K$ (i.e. $\dim_K D_1 = \dim_K D_2 = n^2$), and a degree $n$ field extension $P/K$ admits a $K$-embedding $P \hookrightarrow D_1$ if and only if it admits a $K$-embedding $P \hookrightarrow D_2.$

Let us recall that for a central simple algebra $A$ of degree $n$ over $K$, a field extension $F/K$ is called a {\it splitting field} of $A$ if $A \otimes_K F \simeq M_n(F)$ as $F$-algebras. Furthermore, if $A$ is a division algebra, then the
splitting fields of degree $n$ over $K$ are precisely the maximal
subfields of $A$ (see, e.g., \cite[Theorem 4.4]{DF}). Since splitting fields and/or maximal subfields of a division $K$-algebra $D$ (or, more generally, any finite-dimensional central simple algebra) are at the heart of the analysis of its structure, one is naturally led to ask to what extent these fields actually determine $D$. In the case that one considers {\it all} splitting fields, this question was answered in 1955 by Amitsur \cite{Ami} with the following result:

\vskip2mm

\noindent {\it If $A_1$ and $A_2$
are finite-dimensional central simple algebras over a field $K$ that have the same splitting fields (i.e. a field extension $F/K$ splits $A_1$ if and only if it
splits $A_2$), then the classes $[A_1]$ and $[A_2]$ in the Brauer
group $\Br(K)$ generate the same subgroup: $\langle [A_1] \rangle =
\langle [A_2] \rangle$.}

\vskip2mm
\noindent The proof of this theorem (cf. \cite{Ami}, \cite[Ch. 5]{GiSz}) uses so-called {\it
generic splitting fields}, which have {\it infinite}  degree over
$K$. However, the situation changes dramatically if one allows only finite-dimensional splitting fields, as seen in the following example with cubic division algebras over $\Q.$

%Indeed, let us consider the following example of cubic division algebras over $\Q$.

We first recall the Albet-Brauer-Hasse-Noether Theorem, according to which, for a global field $K$, there is an exact sequence
%The construction relies on the theorem of Albert-Brauer-Hasse-Noether: if $K$ is a global field, then there is an exact sequence
%For the construction, we first reminder the reader that, by the theorem of Albert-Brauer-Hasse-Noether (ABHN), for any global field $K$, we have an exact sequence
\begin{equation}\label{E:ABHN}
\tag{{ABHN}}
0 \to \Br(K) \to \bigoplus_{v \in V^K} \Br (K_v) \stackrel{\sum {\rm inv}_v}{\longrightarrow} \Q/ \Z \to 0,
\end{equation}
where $V^K$ is the set of all places of $K$, $K_v$ denotes the completion of $K$ at $v$, and ${\rm inv}_v$ is the so-called {\it invariant map} giving the isomorphism $\Br(K_v) \simeq \Q/ \Z$ if $v$ is a nonarchimedean place and $\Br(K_v) \simeq \Z/ 2 \Z$ if $v$ is a real place (see, e.g., \cite[Ch. VII, 9.6]{ANT} and \cite[\S 18.4]{Pierce} for number fields, \cite[\S 6.5]{GiSz} for function fields, and \cite{Roq} for a historical perspective). Now fix an integer $r \geqslant 2$, and pick $r$
distinct rational primes $p_1, \ldots , p_r$. Let $\varepsilon =
(\varepsilon_1, \ldots , \varepsilon_r)$ be any $r$-tuple with
$\varepsilon_i \in \{ \pm 1 \}$ such that  $\sum_{i = 1}^r \varepsilon_i
\equiv 0(\mathrm{mod} \: 3)$. By (ABHN),%By the Albert-Brauer-Hasse-Noether Theorem (to be referred to as (ABHN) in
%the sequel --- cf. \cite[Ch. VII, 9.6]{ANT} and \cite[18.4]{Pierce}, as well as \cite{Roq} for a historical perspective), to each such $\varepsilon$, there corresponds a central cubic division algebra
there exists a cubic division algebra
$D(\varepsilon)$ over $\Q$ with the following local invariants
(considered as elements of $\Q/\Z$):
$$ \mathrm{inv}_p \:
D(\varepsilon) \ = \ \left\{ \begin{array}{ccl} \displaystyle
\frac{\varepsilon_i}{3} & , & p = p_i \ \ \text{for} \ i = 1, \ldots
, r; \\  0 & , & p \notin \{p_1, \ldots , p_r\} \ \
(\text{including} \ p = \infty)
\end{array} \right.
$$
Then for any two $r$-tuples $\varepsilon' \neq \varepsilon''$ as
above, the algebras $D(\varepsilon')$ and $D(\varepsilon'')$ have
the same finite-dimensional splitting fields, hence the same maximal
subfields (cf. \cite[18.4, Corollary b]{Pierce}), but are not isomorphic.
Obviously, the
number of such $r$-tuples $\varepsilon$ grows with $r$, so
this method enables one to construct an {\it arbitrarily large} (but
finite) number of pairwise nonisomorphic cubic division algebras
over $\Q$ having the same maximal subfields (at the same time, the cyclic subgroup $\langle [D(\varepsilon)] \rangle$
has order 3).

\vskip2mm
A similar construction can be carried out for division algebras of any degree $d > 2.$ On the other hand, it follows from (ABHN) that a central {\it quaternion} division algebra $D$ over a number field is determined up to isomorphism by its set of maximal subfields (see \S \ref{S:Genus}).
%On the other hand, as we will see in \S \ref{S:Genus}, if one takes $D$ to be a quaternion division algebra over a number field (or, more generally, a global field), then its isomorphism class is determined by the maximal subfields.
Thus, restricting attention to finite-dimensional splitting fields (in particular, maximal subfields) makes the question $(\dag)$ rather delicate and interesting.

For the analysis of division algebras having the same maximal subfields, it is convenient to introduce the following notion.
Suppose $D$ is a finite-dimensional central division algebra over a field $K$. We define the {\it genus} of $D$ as
$$
\gen(D)  = \{ \ [D'] \in \Br(K) \mid  D' \  \text{division algebra having the same maximal subfields as} \ D \}.
$$
Among the various questions that can be asked in relation to this definition, we will focus in this paper on the following two:

\vskip2mm

\noindent {\bf Question 1.} {\it When does $\gen(D)$ reduce to a
single element (i.e. when is $D$ determined uniquely up
to isomorphism by its maximal subfields)?}

\vskip1mm

%\noindent (Notice that this is another way of asking when $D$ is  determined uniquely up
%to isomorphism by its maximal subfields.)

\vskip2mm

\noindent {\bf Question 2.} {\it When is $\gen(D)$ finite?}

\vskip2mm

%Regarding Question 1, note that $\vert \gen(D) \vert = 1$ is
%possible only if the class $[D]$ has exponent 2 in the Brauer group. Indeed, the
%opposite algebra $D^{\mathrm{op}}$ has the same maximal subfields as
%$D$. So, unless $D \simeq D^{\mathrm{op}}$ (which is equivalent to
%$[D]$ having exponent 2), we have $\vert \gen(D) \vert > 1$. At the same time, let us point out that the genus may be larger than $\{[D], [D^{\mathrm{op}}]\}.$ Indeed, suppose $\Delta_1$ and $\Delta_2$ are central division algebras over a field $K$ of degrees $3$ and $5$, respectively, and consider the division algebras $$
%D_1 = \Delta_1 \otimes_K \Delta_2 \ \ \ \ \text{and} \ \ \ \ D_2 = \Delta_1 \otimes_K \Delta^{\mathrm{op}}_2.
%$$
%Then $D_2$ is not isomorphic to either $D_1$ or $D_1^{\mathrm{op}}$, but $D_1$ and $D_2$ have the splitting fields, and hence the same maximal subfields. Thus, $[D_2] \in \gen (D_1).$

We will present the available results on Questions 1 and 2 in \S \ref{S:Genus} ; a general technique for approaching such problems, which is based on the analysis of the ramification of division algebras, will be outlined in \S \ref{S:Ramification}. Next, in \S \ref{S:OtherGenus} we will briefly discuss several other useful notions of the genus, including the {\it local genus} and the {\it one-sided genus}. Finally, in \S \ref{S:GenAlgGp}, we will give an overview of some ongoing work whose aim is to extend the analysis of division algebras with the same maximal subfields to the context of algebraic groups with the same (isomorphism or isogeny classes of) maximal tori.

\vskip5mm

\noindent {\bf Notation.} Given a field $K$ equipped with a discrete valuation $v$, we let $K_v$ denote the completion of $K$ with respect to $v$, $\mathcal{O}_v \subset K_v$ the valuation ring, and $\overline{K}_v$ the corresponding residue field. Furthermore, if $K$ is a number field, $V^K$ will denote the set of all places of $K$ and $V^K_{\infty}$ the subset of archimedean places. Finally, for a field $K$ of characteristic $\neq 2$ and any pair of nonzero elements $a,b \in K^{\times}$, we will let
$
\displaystyle{D = \left( \frac{a, b}{K} \right)}
$
be the associated quaternion algebra, i.e. the 4-dimensional $K$-algebra with basis $1, i, j, k$ and multiplication determined by
$$
i^2 = a, \ \ j^2 = b, \ \ ij = k = -ji.
$$
%Note that this is a central simple algebra over $K$, and, conversely, it is well-known (see, e.g., \cite[Proposition 1.2.1]{GiSz}) that any 4-dimensional central division algebra over $K$ is a quaternion algebra. Finally, if $K$ is a number field, $V^K$ will denote the set of all places of $K$ and $V^K_{\infty}$ the subset of archimedean places.

\section{Motivations}\label{S:Motivations}

In this section, we will describe two sources of motivation that naturally lead one to consider questions
%As we have already mentioned, the analysis of maximal subfields plays a central role in the study of both the additive and multiplicative structures of a division algebra. In this section, we would like to describe two additional sources of motivation that naturally lead one to consider questions
in the spirit of $(\dag).$ The first
exhibits a connection with the theory of quadratic forms, while the second (which for us was actually
was the deciding factor) stems from problems in differential geometry.

%While maximal subfields are certainly crucial for the study of both the additive and multiplicative structures of a division algebra $D$, our initial interest in $(\dag)$ actually stemmed from some questions in geometry, which we will now briefly sketch.

%In addition to the importance of maximal subfields to the study of both the additive and multiplicative structure of a division algebra $D$, we would like to mention

%a couple of additional sources of motivation for the question $(\dag)$, one with ties to the theory of quadratic forms, and the other (which for us was the deciding factor) coming from geometry.

Let $K$ be a field of char $\neq 2.$ To a quaternion algebra
$
\displaystyle{D = \left( \frac{a,b}{K} \right)},
$
%Recall that for any pair of elements $a, b \in K^{\times}$, one defines the {\it quaternion algebra}
%$$
%D = \left( \frac{a,b}{K} \right)
%$$
%as the 4-dimensional $K$-algebra with basis $1, i, j, k$ and multiplication determined by
%$$
%i^2 = a, \ \ j^2 = b, \ \ ij = k = -ji.
%$$
%This is easily seen to be a central simple algebra over $K$, and, conversely, it is well-known (see, e.g., \cite[Proposition 1.2.1]{GiSz}) that any 4-dimensional central division algebra over $K$ is a quaternion algebra.
%Recall that corresponding to any pair of elements $a, b \in K^{\times}$, there exists
%a quaternion algebra $D = \left( \frac{a,b}{K} \right)$ over $K$: by definition, $D$ has a basis consisting of elements $1, i, j, k$ that satisfy the usual relations $i^2 = a$, $j^2 = b$, $ij = k = -ji$, etc. (and, conversely, it is easy to show that any 4-dimensional central division algebra over $K$ is a quaternion algebra --- see, e.g. \cite[Proposition 1.2.1]{GiSz}).
%To such an algebra $D$,
we associate the quadratic form
$$
q(x, y, z) = a x^2 + by^2 - ab z^2.
$$
Notice that, up to sign, this is simply the form that gives the reduced norm of a pure quaternion, from which it follows that for any $d \in K^{\times} \setminus {K^{\times}}^2$, the quadratic extension $K(\sqrt{d})/K$ embeds into $D$ if and only if $d$ is represented by $q$ (see \cite[Remark 1.1.4]{GiSz}).

Now suppose that we have two quaternion division algebras $D_1$ and $D_2$ with corresponding quadratic forms $q_1$ and $q_2$. Then these algebras have the same maximal subfields if
%Then it follows that $D_1$ and $D_2$ have the same maximal subfields if
and only if $q_1$ and $q_2$ represent the same elements over $K$. On the other hand, it is well-known that $D_1$ and $D_2$ are $K$-isomorphic if and only if $q_1$ and $q_2$ are equivalent over $K$ (cf. \cite[\S 1.7, Proposition]{Pierce}) . Thus, $(\dag)$ leads us to the following natural question about quadratic forms:

\vskip2mm

\noindent {\it Suppose $q_1$ and $q_2$ are ternary forms of $\mathrm{det} = -1$ that represent the same elements over $K$. Are $q_1$ and $q_2$ necessarily equivalent over $K$?}

\vskip2mm

\noindent Of course, the answer is \emph{no} for general quadratic forms, but, as the results described in \S \ref{S:Genus} show, it may be \emph{yes} for these special forms in certain situations.

Let us now turn to the geometric questions
%We would now like to give a brief description of some
%geometric questions
dealing with length-commensurable locally symmetric spaces that initially led to our interest in $(\dag)$.
The general philosophy in the study of Riemannian manifolds is that the isometry or commensurability class of a manifold $M$ should
%various important geometric invariants of a manifold $M$, such as its isometry or commensurability class, should
to a significant extent be determined by its {\it length spectrum} $L(M)$ (i.e., the collection of the lengths of all closed geodesics). To put this into perspective, in the simplest case, if $M_1$ and $M_2$ are 2-dimensional Euclidean spheres, then the closed geodesics are the great circles, and clearly these have the same lengths if and only if $M_1$ and $M_2$ are isometric. Furthermore, using the trace formula, one can relate this general idea to the problem
of when two isospectral Riemannian manifolds (i.e. those for which the spectra of the Laplace-Beltrami operators coincide) are isometric; informally, this is most famously expressed by Mark Kac's \cite{Kac} question {\it ``Can one hear the shape of a drum?"}

To make things more concrete, and, at the same time, highlight the connections with $(\dag)$, let us now consider what happens for Riemann surfaces of genus $> 1$ (we refer the reader to \cite{PR1} for a detailed discussion of these questions for general locally symmetric spaces).
%To illustrate the connection of such problems to $(\dag)$, we would like to discuss what happens for Riemann surfaces.
%we would now like to discuss briefly the situation for Riemann surfaces.
Let $\mathbb{H} = \{ x + iy \in \C \:\vert \: y > 0 \}$ be the upper half-plane with the standard hyperbolic metric $ds^2 = y^{-2}(dx^2 + dy^2)$, and equipped with the usual isometric action of $SL_2 (\R)$ by fractional linear transformations.
Let $\pi \colon \mathrm{SL}_2(\R) \to
\mathrm{PSL}_2(\R)$ be the canonical projection. Recall that any compact Riemann surface $M$ of genus $> 1$ can be written as a quotient $\mathbb{H}/ \Gamma$, where $\Gamma \subset \mathrm{SL}_2 (\R)$ is a discrete subgroup containing $\{ \pm 1 \}$
with torsion-free image $\pi(\Gamma)$ (cf., e.g., \cite[Theorem 27.12]{For}).
%the quotient $M =
%\mathbb{H}/\Gamma$ is a Riemann surface, and, conversely, any compact Riemann surface $M$ of genus $g > 1$ is of this %form.
%The general idea in geometry is that, to a significant extent, $M$ is determined by the lengths of its \emph{closed geodesics} (for example, if $M_1$ and $M_2$ are 2-dimensional Euclidean spheres, then the closed geodesics are the great circles; if these have the same lengths, then $M_1$ and $M_2$ are isometric). Furthermore, using the trace formula, this idea can be related to Mark Kac's famous question {\it ``Can you hear the shape of a drum?"} (i.e., does the isospectrality of two Riemannian manifolds imply their isometricity).
%Suppose now that we have a Riemann surface $M = \mathbb{H}/\Gamma$.
It is well-known that closed geodesics in $M$ correspond to nontrivial semisimple elements $\gamma \in \Gamma$. Furthermore, since $\pi(\Gamma)$ is discrete and torsion-free, any semisimple element $\gamma \in \Gamma$, with $\gamma \neq \pm 1$, is automatically hyperbolic, hence $\pm \gamma$ is conjugate
%any semisimple element
%$\gamma \in \Gamma$, with $\gamma \neq \pm 1$, is conjugate
to a matrix of the form
$$
\left( \begin{array}{cc} t_{\gamma} & 0 \\ 0 & t_{\gamma}^{-1} \end{array} \right),
$$
with $t_{\gamma}$ a real number $> 1$.
%$t_{\gamma} > 1$ (note that since $\pi(\Gamma)$ is discrete and torsion-free, any semisimple element in $\Gamma$ is automatically hyperbolic, i.e. $t_{\gamma} \in \R$);
The length of the corresponding closed geodesic $c_{\gamma}$ in $M$ is then computed by the formula
%The precise nature of
%this correspondence is not important for us as we only need
%information about the length: one shows that if $c_{\gamma}$ is a
%closed geodesic in $M$ corresponding to a semi-simple element
%$\gamma \in \Gamma$, $\gamma \neq \pm 1$, then its length is given
%by
\begin{equation}\label{E:length1}
\ell(c_{\gamma}) = \frac{2}{n_{\gamma}} \cdot \vert \log \vert
t_{\gamma} \vert \vert,
\end{equation}
where $n_{\gamma}$ is a certain integer (in fact, a {\it winding number} --- see \cite[\S 8]{PR1} for further details).
%a detailed discussion of the correspondence between semisimple elements and closed geodesics for general locally symmetric spaces).
%where $t_{\gamma}$ is an eigenvalue of $\gamma$ (note that since
%$\pi(\Gamma)$ is discrete and torsion-free, any semi-simple $\gamma
%\in \Gamma$ is automatically hyperbolic, i.e. $t_{\gamma} \in \R$),
%and $n_{\gamma} \geqslant 1$ is an integer which is simply the winding number.
Notice that (\ref{E:length1}) implies that
%for the so-called {\it rational weak length spectrum} $\Q \cdot L(M)$, we have
$$
\Q \cdot L(M) = \Q \cdot \{ \log \vert t_{\gamma}\vert  \ \big\vert \ \gamma \in \Gamma \setminus \{ \pm 1 \}, \ \text{semi-simple} \}
$$
(the set $\Q \cdot L(M)$ is sometimes called the {\it rational weak length spectrum} of $M$). In order to analyze
this geometric set-up using algebraic and number-theoretic techniques, one considers the algebra
%The algebraic aspect of this geometric set-up is based on considering
%Now, to relate this geometric set-up to algebra, one considers
%the subalgebra
$$
D = \Q [\Gamma^{(2)} ] \subset M_2 (\R),
$$
where $\Gamma^{(2)} \subset \Gamma$ is the subgroup generated by squares. It turns out that $D$ is a quaternion algebra whose center is the {\it trace field} $K$ of $\Gamma$, i.e. the subfield generated over $\Q$ by the traces $\Tr(\gamma)$ for all $\gamma \in \Gamma^{(2)}$. Moreover, a noncentral semi-simple element $\gamma \in \Gamma^{(2)}$ gives rise to a maximal commutative \'etale subalgebra $K[\gamma]$ (see \cite[\S 3.2]{MR}). We should point out that this algebra $D$ is an important invariant of the group $\Gamma$; in particular, if $\Gamma$ is an arithmetic group, then $D$ is precisely the algebra involved in its description.

Now let $M_1 = \mathbb{H}/ \Gamma_1$ and $M_2 = \mathbb{H}/ \Gamma_2$ be two (compact) Riemann surfaces, with corresponding quaternion algebras $D_i = \Q[\Gamma_i^{(2)}]$ and trace fields $K_i = Z(D_i)$, for $i = 1,2$. Assume that $M_1$ and $M_2$ are {\it length-commensurable} (L-C), i.e.
$$\Q \cdot L(M_1) = \Q \cdot L(M_2).
$$
One then shows that $K_1 = K_2 = : K$ (see \cite[Theorem 2]{PR1}).\footnotemark \footnotetext{The equality of trace fields is actually proved by studying the relation of {\it weak commensurability} between $\Gamma_1$ and $\Gamma_2$, which one can then relate to the length commensurability of the corresponding Riemann surfaces.} Furthermore, it follows from (\ref{E:length1}) that for any nontrivial semi-simple element $\gamma_1 \in \Gamma_1^{(2)}$, there exists a nontrivial semi-simple element $\gamma_2 \in \Gamma_2^{(2)}$ such that
$$
t_{\gamma_1}^m = t_{\gamma_2}^n
$$
for some integers $m,n \geq 1.$ Consequently, $\gamma_1^m$ and $\gamma_2^n \in M_2 (\R)$ are conjugate, and hence we have an isomorphism  of the corresponding \'etale algebras
$$K[\gamma_1] = K[\gamma_1^m] \simeq K[\gamma_2^n] = K[\gamma_2].$$
Thus, the geometric condition (L-C) translates into the algebraic condition that
{\it $D_1$  and $D_2$ have the same isomorphism classes of maximal \'etale subalgebras intersecting $\Gamma_1^{(2)}$  and  $\Gamma_2^{(2)}$, respectively.}

On the other hand, what one actually wants to prove is that (L-C) implies that $M_1$ and $M_2$ are in fact {\it commensurable} (i.e. have a common finite-sheeted cover, or, equivalently, up to conjugation, the subgroups $\Gamma_1$ and $\Gamma_2$ are commensurable in $GL_2(\R)$). If that is the case, then we {\it necessarily} have $D_1 \simeq D_2$ (see \cite[Corollary 3.3.5]{MR}). So, our problem concerning Riemann surfaces leads to the following question about quaternion algebras:

\vskip2mm

\parbox[t]{16cm}{{\it Let $D_1$ and $D_2$ be two quaternion algebras over the same field $K$, and let $\Gamma_i \subset SL(1, D_i)$, for  $i= 1,2$, be Zariski-dense subgroups with trace field $K$. Assume that $D_1$ and $D_2$ have the same isomorphism classes of maximal \'etale subalgebras that intersect $\Gamma_1$ and $\Gamma_2$, respectively. Are $D_1$ and $D_2$ isomorphic?}}

\vskip2mm

This is a more refined version of our original question $(\dag).$ If $K$ is a number field, then, as we have already mentioned, two quaternion division algebras with the same maximal subfields are isomorphic (see \S \ref{S:Genus} for further details). This fact was used by A.~Reid \cite{Reid} to show that any two {\it arithmetically defined} iso-length spectral Riemann surfaces (i.e. $M_1$ and $M_2$ such that $L(M_1) = L(M_2)$)
are commensurable. The general case is likely to depend on the resolution of questions like the one formulated above.
%While studying the length-commensurability of locally symmetric spaces in \cite{PR1}, G.~Prasad and the second author asked if the condition that a quaternion division algebra is determined by its maximal subfields continues to hold for algebras over
%$K = \Q(x)$; the available results in this direction will be given in the next section.

We would like to conclude our discussion with the following two remarks.
First, it should be pointed out that if $\Gamma$ is non-arithmetic, then it is not uniquely determined by the corresponding quaternion algebra (in fact, there may be infinitely many non-commensurable cocompact lattices having the same associated quaternion algebra, cf. \cite{Vin2}), so an affirmative answer to our question about algebras will not immediately yield consequences for Riemann surfaces. Second, along with the precise (quantitative) version of the question formulated above, one can consider a qualitative version, viz. whether we always have finitely many isomorphism classes of division algebras having the same maximal subfields. Let us note that some finiteness results are available even for non-arithmetic Riemann surfaces. For example, it is known that every class of isospectral compact Riemann surfaces consists of finitely many isometry classes; we recall that according to a well-known conjecture, every class of isospectral surfaces is expected to consist of a {\it single} commensurability class.
%however, at this point it is not known if every class of {\it length-commensurable} Riemann surfaces consists of finitely many commensurability classes (it is conjectured that every class of isospectral surfaces consists of a {\it single} commensurability class).

\section{The genus of a division algebra}\label{S:Genus}

In this section, we will give an overview of the available results on Questions 1 and 2 about the genus of a division algebra that were formulated in \S \ref{S:Introduction}.

%In order to give precise statements of our results concerning division algebras with the same maximal subfields, it will be convenient for us to define the notion of the {\it genus} of a division algebra.

%\vskip2mm

%\noindent {\bf Definition.} {\it Let $D$ be a finite-dimensional central division algebra over a field $K$. The} genus {\it of $D$, denoted $\gen(D)$, is the collection of all classes $[D'] \in \Br(K)$ that are represented by central division
%$K$-algebras $D'$ having the same maximal subfields as $D$.}

%\vskip2mm
%For the statements of our results, we will need the following definition.

%\vskip2mm

%\noindent {\bf Definition.} Let $D$ be a central division
%$K$-algebra of degree $n$. The {\it genus} $\gen(D)$ is the set of
%all classes $[D'] \in \Br(K)$ represented by central division
%$K$-algebras $D'$ having the same maximal subfields as $D$.

%\vskip2mm

%Among the very basic questions pertaining to the genus, we would like to focus on the following two:

%The following are two among the very basic questions pertaining to the genus.

%\vskip2mm

%\noindent {\bf Question 1.} {\it When does $\gen(D)$ reduce to a
%single element (i.e. when is $D$ determined uniquely up
%to isomorphism by its maximal subfields)?}

%\vskip1mm

%(Notice that this is another way of asking when $D$ is  determined uniquely up
%to isomorphism by its maximal subfields.)

%\vskip2mm

%\noindent {\bf Question 2.} {\it When is $\gen(D)$ finite?}

%\vskip2mm

We begin with a couple of remarks pertaining to Question 1. First, let us point out that
$\vert \gen(D) \vert = 1$ is
possible only if the class $[D]$ has exponent 2 in the Brauer group. Indeed, the
opposite algebra $D^{\mathrm{op}}$ has the same maximal subfields as
$D$. So, unless $D \simeq D^{\mathrm{op}}$ (which is equivalent to
$[D]$ having exponent 2), we have $\vert \gen(D) \vert > 1$. At the same time, as the example of cubic algebras over number fields discussed in \S \ref{S:Introduction} shows, the genus may very well be larger than $\{ [D], [D^{\mathrm{op}}] \}.$ Another example of this phenomenon can be constructed as follows.
%the genus may very well be larger than $\{[D], [D^{\mathrm{op}}]\}.$ Indeed,
Suppose $\Delta_1$ and $\Delta_2$ are central division algebras over a field $K$ of relatively prime odd degrees, and consider the division algebras $$
D_1 = \Delta_1 \otimes_K \Delta_2 \ \ \ \ \text{and} \ \ \ \ D_2 = \Delta_1 \otimes_K \Delta^{\mathrm{op}}_2.
$$
Then $D_2$ is not isomorphic to either $D_1$ or $D_1^{\mathrm{op}}$, but $D_1$ and $D_2$ have the same splitting fields. Indeed, since the degrees of $\Delta_1$ and $\Delta_2$ are relatively prime, an extension splits $D_1$ if and only if it splits $\Delta_1$ {\it and} $\Delta_2$; then it also splits $\Delta_2^{\rm op}$, hence $D_2$ (and vice versa). Thus, $D_1$ and $D_2$ have the same maximal subfields, and therefore $[D_2] \in \gen (D_1)$. Note that this example is in fact a consequence of the following general fact:
if $D$ is a central division algebra of degree $n$ over a field $K$, then for any integer $m$ relatively prime to $n$, the class $[D^{\otimes m}] = m [D] \in \Br(K)$ is represented by a central division algebra $D_m$ of degree $n$ which has the same maximal subfields as $D$ (see \cite[Lemma 3.6]{RR}).

%Let \ $\Delta_1$ \ and \
%$\Delta_2$ \ be \ division \ algebras / $K$ \ of \ degrees  $3$ and
%\ $5$. \vskip1mm Consider $$D_1 = \Delta_1 \otimes_K \Delta_2 \ \ \
%\ \text{and} \ \ \ \ D_2 = \Delta_1 \otimes_K
%\Delta^{\mathrm{op}}_2.$$ \vskip1mm {\bf Then} \ \ $D_2 \not\simeq
%D_1$ \ \ nor \ \ $D^{\mathrm{op}}_1$, \ \ \alert{but} \ $D_1$ \ and
%\ $D_2$ \ have \ {\usebeamercolor[fg]{myblue} same \ splitting \
%fields}, \vskip2mm  hence \ \ {\usebeamercolor[fg]{myblue} same \ \
%maximal \ \ subfields} \ \ \ \ $\Rightarrow \ \ \ \ [D_2] \in {\bf
%gen}(D_1).$

As we have already indicated, the theorem of Albert-Brauer-Hasse-Noether (ABHN) enables one to give complete answers to Questions 1 and 2 in the case that $K$ is a number field. The precise statement is as follows.

\begin{prop}\label{P:NF}
Suppose $K$ is a number field and let $D$ be a finite-dimensional central division algebra over $K$.

\vskip1mm

\noindent {\rm (a)} \parbox[t]{16cm}{If $[D] \in \Br(K)$ has exponent 2 (in which case $D$ is a quaternion algebra), then $\vert \gen (D) \vert = 1.$}

\vskip1mm

\noindent {\rm (b)} \parbox[t]{15cm}{The genus $\gen (D)$ is finite for all $D$.}
\end{prop}

%Before addressing Questions 1 and 2 over general fields, let us first consider the situation for number fields, where one can give complete answers to both questions. More precisely, suppose $K$ is a number field. Then
%it makes sense to consider first what happens for number fields. It turns that in this case, one can give complete answers for both questions.
%It turns that in this case, we have definitive answers for both questions.
%More precisely, let $K$ be a global field. Then
%any finite-dimensional central division $K$-algebra $D$ of exponent 2 is a quaternion algebra (this follows from the equality of the period and index over $K$) and $\vert \gen (D) \vert = 1.$ Furthermore, we have $\vert \gen (D) \vert < \infty$ for a division algebra $D$ of arbitrary degree.

\vskip2mm

\noindent Recall that according to (ABHN), we have an injective homomorphism
$$
0 \to \Br(K) \to \bigoplus_{v \in V^K} \Br(K_v),
$$
where
%$V^K$ is the set of all places of $K$, $K_v$ is the completion of $K$ at $v$, and
for each $v \in V^K$,
$$
\Br(K) \to \Br(K_v), \ \ \ [D] \mapsto [D \otimes_K K_v]
$$
is the natural map. We say that a finite-dimensional central division algebra $D$ over a number field $K$ (or its class $[D] \in \Br (K)$) is \emph{unramified} at $v \in V^K$ if its image in $\Br (K_v)$ is trivial, and {\it ramified} otherwise (note that this definition is consistent with the notion of ramification over general fields --- see \S \ref{S:Ramification}, and particularly Example 4.1 below).
The (finite) set of places where $D$ is ramified will be denoted by $R(D)$.
%We would like to point out to the reader that a general definition of ramification, which is consistent with the present one for number fields, will be given in \S \ref{S:Ramification}.
%We denote by $R(D)$ the (finite) set of the places where $D$ is ramified.

\vskip2mm

\noindent {\it Sketch of proof of Proposition \ref{P:NF}.} (a) First, we note that a division algebra $D$ over $K$ of exponent 2 is necessarily a quaternion algebra due to the equality of the exponent and index over number fields, which follows from the Grunwald-Wang Theorem (see, e.g., \cite[Ch. VIII, \S2]{Mil}). Next, we consider the 2-torsion part of (ABHN):
%Now let us sketch the proof of the first statement for quaternion algebras over $K$.
%For this, we work with the 2-torsion part of the exact sequence (\ref{E:ABHN}):
$$
0 \to {}_{2}\Br(K) \to \bigoplus_{v \in V^K} {}_2 \Br(K_v).
$$
Since ${}_2 \Br(K_v)$ is either $\Z/ 2 \Z$ or $0$ for all $v \in V^K$, we see that an algebra $D$ of exponent 2 over $K$ is determined uniquely up to isomorphism by its set of ramification places $R(D)$. Consequently, to prove that $\vert \gen (D) \vert = 1$, it suffices to show that if $D_1$ and $D_2$ are two quaternion division algebras having the same maximal subfields, then $R(D_1) = R(D_2).$ This easily follows from weak approximation, together with the well-known criterion that for $d \in K^{\times} \setminus {K^{\times}}^2$,
\begin{equation}\label{E:emb2}
L = K(\sqrt{d}) \ \ \text{embeds \ into} \ \ D \ \ \
\Leftrightarrow \ \ \ d \notin {K_v^{\times}}^2 \ \ \text{for all} \
\ v \in R(D).
\end{equation}
(cf. \cite[\S18.4, Corollary b]{Pierce}). Indeed, suppose there exists $v_0 \in R(D_1) \setminus R(D_2).$ Then, using the openness of
${K_v^{\times}}^2 \subset K_v^{\times}$ and weak approximation, one
can find $d \in K^{\times} \setminus {K^{\times}}^2$ such that
$$
d \in {K_{v_0}^{\times}}^2 \ \ \text{but} \ \ d \notin
{K_v^{\times}}^2 \ \ \text{for all} \ \ v \in R(D_2).
$$
Then according to (\ref{E:emb2}), the quadratic extension $L =
K(\sqrt{d})$ embeds into $D_2$ but not into $D_1$, a contradiction.

\vskip2mm

\noindent (b)  Let $D$ be a division algebra over $K$ of degree $n > 2$. We now consider the $n$-torsion part of (\ref{E:ABHN}):
$$
0 \to {}_n\Br(K) \to \bigoplus_{v \in V^K} {}_n\Br(K_v).
$$
To establish the finiteness of $\gen(D)$, one first observes
that if $D' \in \gen (D)$, then $R(D') = R(D)$ (cf. \cite[\S18.4, Corollary b]{Pierce}; in fact this is true not just over number fields --- see Lemma \ref{L:=} below).
%, i.e. if $D$ is unramified at $v \in V^K$, then so is $D'$ (in fact, the inclusion is actually an equality -- see \cite[\S18.4, Corollary b]{Pierce}).
Since ${}_n\Br(K_v)$ is $\Z / n \Z$, $\Z/2\Z$ or 0 for all $v \in V^K$ (see Example 4.1(a) below), it follows that
\begin{equation}\label{E:GlobalGen}
\vert \gen (D) \vert \leq \big\vert \bigoplus_{v \in R(D)} {}_n\Br(K_v) \big\vert \leq n^r,
\end{equation}
where $r = \vert R(D) \vert.$
%Let us point out that even though the above estimate can be somewhat improved, we mention this upper bound on the size of the genus since, as we will see below, one has a similar bound for general finitely generated fields.
\hfill $\Box$

\vskip2mm

For a concrete illustration of the argument presented above in the proof of (a), let us consider the following

\vskip1mm

\noindent {\bf Example 3.2.} Let
%Take the following two quaternion division algebras over $\Q$:
$$D_1 = \left( \frac{-1 , 3}{\Q}
\right) \ \ \text{and} \ \ D_2 = \left( \frac{-1 , 7}{\Q} \right).$$
One can check that $R(D_1) = \{2 , 3 \}$ and $R(D_2) = \{2 ,
7\}$, so $D_1$ and $D_2$ are nonisomorphic quaternion division algebras over $\Q$. Clearly, $10 \in {\Q_3^{\times}}^2$ while $10 \notin
{\Q_2^{\times}}^2 , {\Q_7^{\times}}^2$. Thus, according to (\ref{E:emb2}), the field $L = \Q(\sqrt{10})$
embeds into $D_2$ but not into $D_1$. In other words, $D_1$ and $D_2$ are
distinguished by their quadratic subfields. We would like to point out that the recent preprint \cite{LMPT} gives an effective way of producing, for two distinct quaternion algebras over an arbitrary number field, a quadratic field that distinguishes them by putting an explicit bound on its discriminant ({\it loc. cit.}, Theorem 1.3).

%In this example, one was able to find explicitly a quadratic field that distinguishes two given quaternion algebras (or which embeds into one of the given quaternion algebras but not in the other). The recent preprint [.] gives an effective way of producing such a field for any two distinct quaternion algebras over an arbitrary number field by putting an explicit bound on its discriminant ({\it loc. cit.}, Theorem 1.3)

\vskip2mm

\addtocounter{thm}{1}

%We now turn to Question 2. Let $D$ be a division algebra over $K$ of degree $n > 2$. We need the $n$-torsion part of (\ref{E:ABHN}):
%$$
%0 \to {}_n\Br(K) \to \bigoplus_{v \in V^K} {}_n\Br(K_v).
%$$
%To establish the finiteness of $\gen(D)$, one first shows
%that if $D' \in \gen (D)$, then $R(D') \subset R(D)$, i.e. if $D$ is unramified at $v \in V^K$, then so is $D'$ (in fact, the inclusion is actually an equality). It follows that
%\begin{equation}\label{E:GlobalGen}
%\vert \gen (D) \vert \leq \big\vert \bigoplus_{v \in R(D)} {}_n\Br(K_v) \big\vert \leq n^r,
%\end{equation}
%where $r = \vert R(D) \vert.$ Though this estimate on the size of the genus can be somewhat improved, we mention it here since, as we will see later, this is the form in which it can be extended to other finitely generated fields.

%the result in this form since, as we will see later, it can be extended to other finitely generated fields.

It is now natural to ask whether (and to what extent) these results for the genus carry over to general fields. Namely, can we expect the genus to be {\it trivial} for a quaternion division algebra, and {\it finite} for any finite-dimensional division algebra, over an {\it arbitrary} field $K$? It turns out that the answer is \emph{no} in both cases. Several people, including Rost, Schacher, Wadsworth, etc., have given a construction of quaternion algebras with nontrivial genus over certain very large fields. We refer the reader to \cite[\S 2]{GS} for the full details, and only sketch the main ideas here.

%Several people, including Wadsworth, Schacher, Rost, Saltman, Garibaldi, have given
%(probably semi-independently) a construction of nonisomorphic quaternion algebras over certain very large fields with the same quadratic subfields.

%The basic idea is the following (see \cite{GS} for the full details).
Let $D_1$ and $D_2$ be two nonisomorphic quaternion division algebras $D_1$
and $D_2$ over a field $k$ of characteristic $\neq 2$ that have a
common quadratic subfield (e.g., one can take $k = \Q$ and the quaternion algebras
$\displaystyle D_1$ and
$\displaystyle D_2$ considered in Example 3.2 above). If $D_1$ and
$D_2$ already have the same quadratic subfields, we are done.
Otherwise, there exists a quadratic extension $k(\sqrt{d})$ that
embeds into $D_1$ but not into $D_2$. Applying some general results on quadratic forms to the norm forms of $D_1$ and $D_2$, one shows that there exists a field extension
%using certain properties
%of quadratic forms (cf. ...), one shows
%that there exists an extension
$k^{(1)}$ of $k$ (which is the function field of an appropriate quadric) such that

\vskip2mm

\noindent $\bullet$ $D_1 \otimes_k k^{(1)}$ and  $D_2 \otimes_k
k^{(1)}$ are non-isomorphic division algebras over $k^{(1)}$, \ {\bf
but}

\vskip1mm

\noindent $\bullet$ $k^{(1)}(\sqrt{d})$ embeds into $D_2 \otimes_k
k^{(1)}$.

\vskip2mm

\noindent One deals with other subfields one at a time by applying the same procedure to the algebras obtained from $D_1$ and $D_2$ by extending scalars to the field extension constructed at the previous step.
%One deals with other subfields, in the algebras obtained
%from $D_1$ and $D_2$ by applying the extension of scalars built at
%the previous step of the construction, one at a time, in a similar
%fashion.
This process generates an ascending chain of fields
$$
k^{(1)} \subset k^{(2)} \subset k^{(3)} \subset \cdots ,
$$
and we let $K$ be the union (direct limit) of this chain. Then $D_1
\otimes_k K$ and $D_2 \otimes_k K$ are non-isomorphic quaternion
division $K$-algebras having the same quadratic subfields; in
particular $\vert \gen(D_1 \otimes_k K) \vert > 1$. Note that the
resulting field $K$ has infinite transcendence degree over $k$, and, in particular, is
{\it infinitely generated}. Furthermore, a modification of the above construction
(cf. \cite{Meyer}) enables one to start with an
infinite sequence $D_1, D_2, D_3, \ldots$ of quaternion division algebras over
a field $k$ of characteristic $\neq 2$ that are pairwise
non-isomorphic but share a common quadratic subfield (e.g., one can
take $k = \Q$ and consider the family of algebras of the form
$\displaystyle \left(\frac{-1 , p}{\Q} \right)$ where $p$ is a prime
$\equiv 3(\mathrm{mod} \: 4)$), and then build an infinitely
generated field extension $K/k$ such that the algebras $D_i
\otimes_k K$ become pairwise non-isomorphic quaternion division algebras with
any two of them having the same quadratic subfields. In particular, this yields an example of quaternion division algebras with {\it infinite} genus. More recently, a similar approach was used in \cite{Tikh} to show that for any prime $p$, there exists a field $K$ and a central division algebra $D$ over $K$ of degree $p$ such that $\gen(D)$ is infinite.

Thus, we see that while Questions 1 and 2 can be answered completely and in the affirmative over number fields, and more generally, over global fields, these questions become nontrivial over arbitrary fields. In fact, until quite recently, very little was known about the situation over fields other than global.
As we mentioned in \S \ref{S:Motivations}, the triviality of the genus for quaternion division algebras $D$ over number fields has consequences for arithmetically defined Riemann surfaces. In connection with their work on length-commensurable locally symmetric spaces in \cite{PR1}, G.~Prasad and the second-named author asked whether $\vert \gen (D) \vert = 1$ for any central quaternion division algebra $D$ over $K = \Q(x)$.\footnotemark \footnotetext{Such questions have certainly been around informally for quite some time, at least since the work of Amitsur \cite{Ami} mentioned in \S \ref{S:Introduction}, but to the best of our knowledge, this was the first ``official" formulation of a question along these lines, though the terminology of the genus was introduced later.}
%(of course, the notion of the genus was not used at the time, but, to the best of our knowledge, this was the first ``official" formulation of a question along the lines of Question 1).
This question was answered in the affirmative by D.~Saltman, and later, in joint work with S.~Garibaldi, it was shown that any quaternion division algebra over $k(x)$, where $k$ is an arbitrary number field, has trivial genus (in fact, they considered more generally so-called {\it transparent} fields of characteristic $\neq 2$ --- see \cite{GS}). Motivated by this result, we showed in \cite{RR} that for a given field $k$ of characteristic $\neq 2$, the triviality of the genus for quaternion division algebras over $k$ is a property that is stable under purely transcendental extensions. Subsequently, we established a similar {\it Stability Theorem} for algebras of exponent 2. The precise statements are given in the following.
%the following {\it Stability Theorem} for algebras of exponent 2.
%The affirmative answer to this question was given by D.~Saltman, and later, in joint work with S.~Garibaldi, they proved that the same remains true over $k(x)$, where $k$ is an arbitrary number field, as well as in some other situations (cf. \cite{GS}). Recently, we have proved in \cite{CRR2} the following {\it Stability Theorem}
%for algebras of exponent 2 (the case of
%quaternion algebras was considered earlier in \cite{RR}).
\begin{thm}\label{T:StabThm}
{\rm (see \cite[Theorem A]{RR} and \cite[Theorem 3.5]{CRR2})} Let $k$ be a field of
characteristic $\neq 2$.

\vskip1mm

\noindent {\rm (a)} \parbox[t]{16cm}{If $\vert \gen(D) \vert = 1$ for any quaternion division algebra $D$ over $k$, then $\vert \gen(D') \vert = 1$ for any quaternion division algebra $D'$ over the field of rational functions $k(x).$}

\vskip1mm

\noindent {\rm (b)} \parbox[t]{16cm}{If $\vert \gen(D) \vert = 1$ for any
central division $k$-algebra $D$ of exponent 2, then the same
property holds for any central division algebra of exponent 2 over $k(x)$.}
\end{thm}

\vskip2mm

As an immediate consequence, we have

\begin{cor}\label{C:1}
If $k$ is either a number field or a finite field of $\mathrm{char}
\neq 2$, and $K = k(x_1, \ldots , x_r)$ is a purely transcendental
extension, then for any central division $K$-algebra $D$ of exponent
2, we have $\vert \gen(D) \vert = 1$.
\end{cor}

\vskip2mm

\noindent {\bf Remark 3.5.} In \cite{Common}, the Stability Theorem has been generalized to function fields of Severi-Brauer varieties of algebras of odd degree.
%In joint work with E.~Matzri, the second author has obtained a generalization of the Stability Theorem to function fields of Severi-Brauer varieties of algebras of odd degree.

%There is now a generalization of the Stability Theorem to function fields of Severi-Brauer varieties of algebras of odd degree (this is joint work of the second author with E.~Matzri).

%\vskip1mm

%\noindent (b) At this point, we do not have a single example of a quaternion algebra over a finitely-generated field with nontrivial genus.

\vskip2mm

Let us now turn to Question 2. As the above discussion shows, one cannot hope to have a finiteness result for the genus over completely arbitrary fields. Nevertheless, we have obtained the following finiteness statement over finitely-generated fields.

%Regarding Question 2, we have the following finiteness statement over finitely-generated fields.

\addtocounter{thm}{1}

\begin{thm}\label{T:Finite1}
{\rm (\cite{CRR1}, Theorem 3)} Let $K$ be a finitely generated
field. If $D$ is a central division $K$-algebra of exponent prime to
$\mathrm{char} \: K$, then $\gen(D)$ is finite.
\end{thm}

\vskip2mm

\noindent {\bf Remark 3.7.} (a) It is still an open question whether there exist quaternion division algebras over finitely generated fields with {\it nontrivial} genus.

\vskip1mm

\noindent (b) In \cite{KM}, D.~Krashen and K.~McKinnie have studied division algebras having the same finite-dimensional {\it splitting fields} (see \S\ref{S:OtherGenus} below for the associated notion of the genus).
%defined the genus $\gen'(D)$ of a central division
%$K$-algebra $D$ as the collection of $[D'] \in \Br(K)$ having the
%same finite-dimensional splitting fields as $D$ (clearly $\gen'(D)
%\subset \gen(D)$). They showed that if $k$ is a field such that $\gen'(D)$ is finite for all central division algebras $D$ over $k$ of prime exponent $p \neq \mathrm{char} \: k$ %(e.g., if $k$ is a number field), then one can give an upper bound on the size of $\gen' (D')$ for a central division algebra $D'$ of exponent $p$ over $K = k(x)$ (see %\cite[Theorem 2.2]{KM} for the precise statement).

%Their Theorem 2.2 provides estimates for $\vert
%\gen'(D') \vert$ of a central division algebra $D$ over $K = k(x)$
%of a prime exponent $p \neq \mathrm{char} \: k$ assuming that

%A question that still remains open is whether there exist quaternion division algebras over \emph{finitely generated fields} with {\it nontrivial genus}.

%One of the questions about the genus of a division algebra that remains open
%after Theorems \ref{T:StabThm} and \ref{T:Finite1} is whether one
%can find a quaternion division algebra {\it over a finitely
%generated field} of characteristic $\neq 2$ with {\it nontrivial
%genus}.

\section{Ramification of division algebras}\label{S:Ramification}

In this section, we will describe some of the ideas involved in the proofs
of Theorems \ref{T:StabThm} and \ref{T:Finite1}. In broad terms, the arguments rely on many of the same general considerations that were already encountered in our discussion of the genus over number fields. More precisely, even though there is no direct counterpart of (ABHN) for arbitrary (finitely generated) fields, the analysis of \emph{ramification} of division algebras nevertheless again plays a key role.

%the arguments rely on many of the same considerations that came up in the discussion of global fields. More precisely, even though we do not have such fundamental results as (ABHN) for arbitrary fields, a key role is still played by the analysis of \emph{ramification} of division algebras.

We begin by recalling the standard set-up that is used in the analysis of ramification of division algebras. Let $K$ be a field equipped with a discrete valuation $v$, denote by %, $v$ a discrete valuation of $K$, and $K_v$ the corresponding completion with residue field $\overline{K}_v.$ Denote by
$\cG^{(v)} = \Ga (\overline{K}_v^{\rm sep}/ \overline{K}_v)$ the absolute Galois group of the residue field $\overline{K}_v$ of $K_v$, and
fix an integer $n > 1.$ If either $n$ is prime to ${\rm char} \ \overline{K}_v$ or $\overline{K}_v$ is perfect, there exists a {\it residue map}
\begin{equation}\label{E:LocalRes}
r_v \colon {}_n\Br(K_v) \to \Hom (\cG^{(v)}, \Z/ n \Z),
\end{equation}
where the group on the right is the group of {\it continuous} characters mod $n$ of $\cG^{(v)}$ (see \cite[\S 10]{Salt}). Furthermore, it is known that $r_v$ is surjective.
%it can be shown that $r_v$ is in fact surjective.
A division algebra $D$ over $K_v$ of exponent $n$ (or its corresponding class $[D] \in {}_n\Br(K_v)$) is said to be \emph{unramified} if $r_v ([D]) = 0$, and the {\it unramified Brauer group at $v$} is defined as
$$
{}_n\Br(K_v)_{\{v \}} = \ker r_v.
$$
Thus, by construction, we have an exact sequence
\begin{equation}\label{E:RamLocalSeq}
0 \to {}_n\Br(K_v)_{\{v\}} \to {}_n\Br(K_v) \stackrel{r_v}{\longrightarrow}  \Hom (\cG^{(v)}, \Z/ n \Z) \to 0.
\end{equation}
It can be shown that there is an isomorphism ${}_n\Br (K_v)_{\{v \}} \simeq {}_n\Br(\overline{K}_v),$ and moreover, the latter group is naturally identified with ${}_n\Br (\mathcal{O}_v),$ where $\mathcal{O}_v$ is the valuation ring of $K_v$ (see \cite[Theorem 3.2]{W2} and the subsequent discussion). In other words, the algebras that are unramified at $v$ are precisely the ones that arise from {\it Azumaya algebras} over $\mathcal{O}_v.$

Next, composing the map in (\ref{E:LocalRes}) with the natural homomorphism ${}_n\Br (K) \to {}_n \Br(K_v)$, we obtain a residue map
\begin{equation}\label{E:LocalRes1}
\rho_v \colon {}_n\Br(K) \to \Hom (\cG^{(v)}, \Z/ n \Z),
\end{equation}
and again one says that a division algebra $D$ over $K$ is {\it unramified} if $\rho_v([D]) = 0.$  For fields other than local, to get information about ${}_n\Br(K)$, one typically uses not just one valuation of $K$, but rather a suitable set of valuations.
%In the analysis of ramification, it is most useful to consider not just a single valuation of $K$, but rather an entire collection of valuations, with some prescribed properties.
For this purpose, it is convenient to make the following definition. Fix an integer $n > 1$ and suppose $V$ is a set of discrete valuations of $K$ such that the residue maps $\rho_v$ exist for all $v \in V.$ We define the {\it unramified ($n$-torsion) Brauer group with respect to $V$} as
$$
{}_n\Br(K)_V = \bigcap_{v \in V} \ker \rho_v.
$$

\vskip2mm

\noindent {\bf Example 4.1.} (a) Let $p$ be a prime, $K = \Q_p$, and $v = v_{p}$ be the corresponding $p$-adic valuation.
Then, since $\mathbb{F}_p$ is perfect and $\Br (\mathbb{F}_p) = \{ 0 \},$ it follows from (\ref{E:RamLocalSeq}) that for any $n > 1,$ we have an isomorphism
$$
{}_n\Br (\Q_p) \stackrel{r_v}{\simeq} \Hom (\Ga (\bar{\mathbb{F}}_p/ \mathbb{F}_p), \Z / n \Z ) = \Hom_{\text{cont}} (\widehat{\Z}, \Z/ n \Z) \simeq  \Z / n \Z.
$$
Taking the direct
limit over all $n$, we obtain an isomorphism $\Br(\Q_p) \simeq \Q/\Z$, which is precisely the invariant map from local class field theory. This description of the Brauer group immediately extends to any finite extension of $\Q_p$. An important point here is that only the trivial algebra is unramified  over such a field.
%This yields an isomorphism $\Br (\Q_p) \simeq \Q / \Z,$ which can be used to recover the usual invariant map from local class field theory. The situation is completely analogous for any local field. In particular, we see that a division algebra $D$ over $K$ is unramified at $v$ if and only if $[D] = 0 \in \Br(K).$

\vskip1mm

\noindent (b) Let $K$ be a number field. It follows from part (a) that a finite-dimensional central division algebra $D$ over $K$ is unramified at a nonarchimedean place $v \in V^K$ if and only if its image in $\Br(K_v)$ is trivial (which is consistent with the definition that we used in \S \ref{S:Genus}). Combining this with (ABHN), we see that if
$S \subset V^K$ is a finite set containing $V^K_{\infty}$, then for $V = V^K \setminus S$ and any integer $n > 1$, the unramified Brauer group
$$
{}_n\Br(K)_V
$$
is finite.
%(notice that it follows from (a) that this definition of ramification coincides with the one used in our discussion of Proposition \ref{P:NF}).

\vskip1mm

\noindent (c) Let $C$ be a smooth connected projective curve over a field $k$ and $K = k(C)$ be the function field of $C$. For each closed point $P \in C$, we have a discrete valuation $v_P$ on $K$. The set of discrete valuations
$$
V_0 = \{ v_P \mid P \in C \ \text{a closed point} \}
$$
is usually called the set of {\it geometric places} of $K$ (notice that these are precisely the discrete valuations of $K$ that are trivial on $k$). If $n > 1$ is an integer that is relatively prime to $\text{char} \ k$, then for each $v_P \in V_0$,  we have a residue map $\rho_P = \rho_{v_P} \colon {}_n\Br(K) \to H^1 (G_P, \Z / n \Z)$, where $G_P$ is the absolute Galois group of the residue field at $P$.  In this case, the corresponding unramified Brauer group ${}_n\Br(K)_{V_0}$ will be denoted, following tradition, by ${}_n\Br(K)_{\rm ur}.$
Furthermore, if $k$ is a perfect field and $C$ is geometrically connected, then the above residue map $\rho_P$ extends to a map $\Br(K) \to H^1(G_P , \Q/\Z)$
%where $C_0$ denotes the set of closed points of $C$ and $G_P$ is the absolute Galois group of the residue field at $P$,
defined on the entire Brauer group. We can then consider
$$
\Br(K) \stackrel{\oplus \rho_P}{\longrightarrow} \bigoplus_{P \in C_0} H^1 (G_P, \Q/\Z),
$$
where $C_0$ denotes the set of closed points of $C$.
The kernel $\displaystyle{\ker (\oplus \rho_P) = : \Br(K)_{\rm ur}}$  is known to coincide with the Brauer group $\Br(C)$ defined either in terms of Azumaya algebras or in terms of \'etale cohomology (see \cite[\S 6.4]{GiSz} and \cite{Lich}).  Then the group ${}_n\Br(K)_{\mathrm{ur}}$ is precisely the $n$-torsion of $\Br(C)$.
%Furthermore, if $k$ is a perfect field and $C$ is geometrically connected,
%there exists a residue map
%$$
%\Br(K) \stackrel{\oplus \rho_P}{\longrightarrow} \bigoplus_{P \in C_0} H^1 (G_P, \Q/\Z),
%$$
%where $C_0$ denotes the set of closed points of $C$ and $G_P$ is the absolute Galois group of the residue field at $P$, and for each $n > 1$, we obtain ${}_n\Br(K)_{\rm ur}$ as the %kernel of the restriction of $\oplus \rho_P$ to the $n$-torsion. The kernel $$\ker (\oplus \rho_P) = : \Br(K)_{\rm ur}$$ is known to coincide with the geometric Brauer group %$\Br(C)$ of $C$ that is defined in terms of \'etale cohomology (see \cite[\S 6.4]{GiSz} and \cite{Lich}).

\addtocounter{thm}{1}

\vskip4mm

%Returning to the general set-up, if we compose the map in (\ref{E:LocalRes}) with the natural homomorphism ${}_n\Br (K) \to {}_n \Br(K_v)$, we obtain a residue map
%$$
%\rho_v \colon {}_n\Br(K) \to \Hom (\cG^{(v)}, \frac{1}{n} \Z/ \Z),
%$$
%and again one says that a division algebra $D$ over $K$ is {\it unramified} if $\rho_v([D]) = 0.$

Returning to the general set-up, we would now like to mention a result that describes the ramification behavior of division algebras lying in the same genus. Let $K$ be a field equipped with a discrete valuation $v$ and $n > 1$ be an integer that is relatively prime to the characteristic of the residue field $\overline{K}_v.$ We have the following.

\begin{lemma}\label{L:=}
Let $D$ and $D'$ be central division $K$-algebras such that $$[D] \in
{}_n\Br(K) \ \  \text{and} \ \  [D'] \in \gen(D) \cap {}_n\Br(K).$$ Let $\chi_v$ and $\chi_v' \in \mathrm{Hom}(\mathcal{G}^{(v)} ,
\Z/n\Z)$ denote the images of $[D]$ and $[D']$, respectively, under the residue map $\rho_v$.
Then
$$
\ker \chi_v = \ker \chi_v'.
$$
%for all $v \in V$.
In particular, $D$ is unramified at $v$ if and only if $D'$ is.
\end{lemma}

\noindent %In particular, this implies that if $D$ is unramified at $v$, then so is $D'$, and vice versa.
For the argument, let us recall that if $\mathcal{K}$ is a field that is complete with respect to a discrete valuation $v$ and $\mathcal{D}$ is a finite-dimensional central division algebra over $\mathcal{K}$, then $v$ extends uniquely to a discrete valuation $\tilde{v}$ on $\mathcal{D}$ (see \cite[Ch. XII, \S2]{Se}, \cite{W2}). Furthermore, the corresponding valuation ring
$\cO_{\mathcal{D}}$ has a unique maximal 2-sided ideal
$\mathfrak{P}_{\mathcal{D}}$ (the valuation ideal), and the quotient
$\overline{\mathcal{D}} =
\cO_{\mathcal{D}}/\mathfrak{P}_{\mathcal{D}}$ is a
finite-dimensional division (but not necessarily central) algebra,
called the {\it residue algebra}, over the residue field
$\overline{\mathcal{K}}$.

\begin{proof}(Sketch) Write
%The proof proceeds in three steps. Let $K_v$ be the completion of $K$ with respect to $v$ and write
$$
D \otimes_K K_v = M_{\ell}(\mathcal{D}) \ \ \ \text{and} \ \ \ D'
\otimes_K K_v = M_{\ell'}(\mathcal{D}'),
$$
where $\mathcal{D}$ and $\mathcal{D}'$ are central division algebras over $K_v$. First, one shows that
%by \cite[Corolary 2.4]{RR}, we have
$\ell = \ell'$ and $\mathcal{D}$ and $\mathcal{D}'$ have the same maximal subfields (\cite[Corollary 2.4]{RR}). Next, one verifies that the centers $\mathcal{E}$ and $\mathcal{E}'$ %denoting by $\mathcal{E}$ and $\mathcal{E}'$ the centers
of the residue algebras
$\overline{\mathcal{D}}$
and $\overline{\mathcal{D}}'$, respectively, coincide (\cite[Lemma 2.3]{CRR2}).
%we infer from \cite[Lemma 2.3]{CRR2} that $\mathcal{E} = \mathcal{E}'$.
Finally, using the fact that $\ker \chi_v$ and $\ker \chi'_v$ are precisely the subgroups of $\mathcal{G}^{(v)}$ corresponding to $\mathcal{E}$ and $\mathcal{E}'$, respectively (cf. \cite[Theorem 3.5]{W2}), we conclude that
$$
\ker \chi_v = \ker \chi'_v,
$$
as needed.
%On the other hand, as we already
%mentioned in the proof of Lemma \ref{L:R2}, $\mathrm{Ker}\: \chi_v$
%and $\mathrm{Ker}\: \chi'_v$ are precisely the subgroups of
%$\mathcal{G}^{(v)}$ corresponding to $\mathcal{E}$ and
%$\mathcal{E}'$, respectively (cf. \cite[Theorem 3.5]{W2}). So, our
%claim follows.
\end{proof}

We are now in a position to outline the proof of part (a) of Theorem \ref{T:StabThm} (see the proof of Theorem A in \cite{RR} for full details). Let $k$ be a field of characteristic $\neq 2$ such that $\vert \gen(D) \vert =1$ for every quaternion division algebra $D$ over $k.$ Viewing $k(x)$ as the function field of $\mathbb{P}_k^1$, we consider the following segment of {\it Faddeev's Exact Sequence} (see \cite[Corollary 6.4.6]{GiSz}):
\begin{equation}\label{E:Faddeev}
0 \to {}_2\Br(k) \to {}_2\Br(k(x)) \stackrel{\oplus \rho_P}{\longrightarrow} \bigoplus_{P \in (\mathbb{P}_k^1)_0} H^1 (G_P, \Z / 2 \Z).
\end{equation}
%(Notice that one can view the closed points of $\mathbb{P}_k^1$ concretely as corresponding to the irreducible polynomials $p(x) \in k[x]$, together with the local parameter $t^{-1}$ at $P = \infty$.)
Suppose now that $D$ and $D'$ are quaternion division algebras over $k(x)$ that have the same maximal subfields. Then it follows from Lemma \ref{L:=}, together with the fact that $[D]$ and $[D']$ have exponent 2 in $\Br(k(x))$, that
$$
\rho_P([D]) = \rho_P([D'])
$$
for all $P \in (\mathbb{P}_k^1)_0.$ Consequently, (\ref{E:Faddeev}) yields that
$$
[D] = [D'] \cdot [\Delta \otimes_k k(x)]
$$
for some central quaternion algebra $\Delta$ over $k$ (see \cite[Corollary 4.2]{RR}). A specialization argument, which relies on the assumption that quaternion division algebras over $k$ have trivial genus, then shows that $[\Delta] = 0$, and hence $D \simeq D'$. Consequently $\vert \gen (D) \vert = 1$, as required.
%(see the proof of Theorem~A in \cite{RR} for full details).

Now we would like to indicate the main elements of the proof of Theorem \ref{T:Finite1}, which is also based on an analysis of ramification. Let $K$ be a finitely generated field and fix an integer $n > 1$ relatively prime to $\text{char} \ K.$ We will need to consider sets $V$ of discrete valuations of $K$ satisfying the following two properties:

%\vskip2mm

%\noindent $\text{(I)}$ {\it for any $v \in V$, the characteristic of the residue field $\overline{K}_v$ is prime to $n$;}

\vskip2mm

\noindent $\text{(I)}$ {\it for any $a \in K^{\times}$, the set
$
V(a) := \{ v \in V \mid v(a) \neq 0 \}
$
is finite;}

\vskip2mm

\noindent $\text{(II)}$ {\it for any $v \in V$, the characteristic of the residue field $\overline{K}_v$ is prime to $n$.}

\vskip2mm

\noindent Note that (II) ensures the existence of a residue map
$$
{}_n\Br(K) \stackrel{\rho_v}{\longrightarrow} {\rm Hom}(\mathcal{G}^{(v)}, \Z / n \Z)
$$
for each $v \in V$.
%that the residue maps $\rho_v$ $(v \in V)$ are defined on algebras of degree (or exponent) $n.$
When considering the case of number fields in \S \ref{S:Genus}, we observed that the set of ramification places is finite for any division algebra $D$, which, with the help of (ABHN), led
%with the help of (ABHN) led us
to the upper bound (\ref{E:GlobalGen}) on the size of $\gen(D)$. Over general fields, condition (I) again guarantees the finiteness of the set
$$
R(D) = R(D, V) = \{ v \in V \mid \rho_v([D]) \neq 0 \}
$$
of ramification places for any division algebra $D$ of degree $n$ over $K$ (see \cite[Proposition 2.1]{CRR2} for a slightly more general statement).
To obtain an analogue of (\ref{E:GlobalGen}), we argue as follows, using Lemma \ref{L:=}.
%Next, recall that according to (ABHN), any division algebra $D$ ramifies at only a finite number of places, which yielded, in particular, the upper bound (\ref{E:GlobalGen}) on the size of $\gen(D).$ Over a general finitely generated field $K$, one easily shows that condition (II) guarantees the finiteness of the set
%that over a number field, (ABHN) implies that any division algebra $D$ is ramified at only a finite number of places, which led to the upper bound (\ref{E:GlobalGen}) on the size of $\gen(D)$. Over a general finitely generated field $K$, it is easy to show that condition (II) yields the finiteness of the set
%recall that for a global field, our estimate on the size of the genus of a division algebra $D$ depended on the number of ramification points of $D$ see -- (\ref{E:GlobalGen}). Thus, a similar result over
%in order to be able to obtain a similar result over
%a general field $K$ is possible only if we can ensure that for any division algebra $D$ over $K$ of degree $n$, the set
%In order to use the same strategy in the general case, we need to impose a condition on $V$ such that the set
%$$
%R(D) = R(D, V) = \{ v \in V \mid \rho_v([D]) \neq 0 \}
%$$
%of ramification places for any division algebra $D$ of degree $n$ over $K$ (see \cite[Proposition 2.1]{CRR2} for a slightly more general statement). This has the following consequence.
Suppose that $D$ is a central division algebra over $K$ of degree $n$ and let $[D'] \in \gen(D).$ For $v \in V$, we set
$$
\chi_v = \rho_v([D]) \ \ \ \text{and} \ \ \ \chi'_v = \rho_v([D']);
$$
recall that by Lemma \ref{L:=}, we have $\ker \chi_v = \ker \chi'_v.$ Notice that if the character $\chi_v$ has order $m \vert n$, then any character $\chi'_v$ of
$\mathcal{G}^{(v)}$ with the same kernel can be viewed as a
faithful character of the cyclic group
$\mathcal{G}^{(v)}/\ker \chi_v$ of order $m$. Consequently, there are
$\varphi(m)$ possibilities for $\chi'$, and therefore
\begin{equation}\label{E:SizeImageGenus}
\vert \rho_v(\gen(D)) \vert \leqslant \varphi(m)
\leqslant \varphi(n)
\end{equation}
for any $v \in V$ (as $m$ divides $n$), and $$\rho_v(\gen(D)) = \{1\}$$ if $\rho_v([D]) = 1$. Letting $\rho = (\rho_v)_{v \in V}$ be the direct sum of the residue maps for all $v \in V$, it follows that
\begin{equation}\label{E:SizeImageGenus1}
\vert \rho(\gen(D)) \vert \leq \varphi(n)^r,
\end{equation}
where $r = \vert R(D) \vert.$ Therefore, if $\ker \rho = {}_n\Br(K)_V$ is finite, we obtain the estimate
\begin{equation}\label{E:SizeImageGenus2}
\vert \gen(D) \vert \leq \varphi(n)^r \cdot \vert {}_n\Br(K)_V \vert.
\end{equation}

\noindent Thus, the proof of Theorem \ref{T:Finite1} boils down to establishing the finiteness of the unramified Brauer group with respect to an appropriate set of discrete valuations, which is the subject matter of the next result.
%This can be done for any finitely generated field according to the following result.
%Thus, establishing the finiteness of $\gen(D)$ boils down to finding a set of discrete valuations $V$ of $K$ satisfying (I) and (II) and such that the unramified Brauer group ${}_n\Br(K)_V$ is finite.
\begin{thm}\label{T:FiniteUnram}{\rm (\cite{CRR1}, Theorem 8)}
Let $K$ be a finitely generated field and $n > 1$ be an integer prime to $\mathrm{char}~K.$ Then there exists a set $V$ of discrete valuations of $K$ that satisfies conditions $\mathrm{(I)}$ and $\mathrm{(II)}$ introduced above and for which ${}_n\Br(K)_V$ is finite.
\end{thm}

Our proof of this theorem, which will be outlined below, relies on the analysis of the exact sequence for the Brauer group of a curve. Subsequently, it was pointed out to us by J.-L. Colliot-Th\'el\'ene that the finiteness statement could also be derived from certain results in \'etale cohomology. More precisely, in this argument, we present
%There are \emph{two} proofs of this theorem. One proof was communicated to us by J.-L. Colliot-Th\'el\`ene and is sketched in \cite{CRR1}. In this argument, we present
$K$ as the field of rational functions on a smooth
arithmetic scheme $X$, with $n$ invertible on $X$. Then one uses Deligne's finiteness theorems for the \'etale cohomology of constructible sheaves (Theorem 1.1 of the chapter
``Th\'eor\`emes de finitude" in \cite{Del}) to show that in this case, the $n$-torsion of the \'etale Brauer group is finite. On the other hand, by Gabber's purity theorem (see
\cite{Fuj} for an exposition of Gabber's proof, and also \cite[p.
153]{CT-S} and \cite[discussion after Theorem 4.2]{CT-Bour}
regarding the history of the question), the latter group coincides with the unramified Brauer group of $K$ with respect to the set $V$ of discrete valuations of $K$ associated with the prime divisors on $X$ (see \cite{CRR1} for more details). We will use the term {\it divisorial} to describe such a set of valuations. Note that this argument allows quite a bit of flexibility
%One advantage of this argument is the flexibility
in the choice of $V$: for example, $X$ can be replaced with an open subscheme, which allows us to delete from $V$ any \emph{finite} set. This flexibility somewhat simplifies the proof of the finiteness of the genus. Indeed, for a given division algebra $D$, we can choose $V$ so that $D$ is unramified at all places of $V$. Then we do not need the full strength of Lemma \ref{L:=}, but only the fact that any $[D'] \in \gen (D)$ is also unramified at all $v \in V$ (see Theorem \ref{T:Main1} below for an analogue of this for arbitrary algebraic groups). This immediately leads to the conclusion that
$$
\vert \gen (D) \vert \leq \vert {}_n\Br(K)_V \vert < \infty.
$$
The main disadvantage here is that this argument does not give any {\it explicit} estimates on the size of the genus.

On the other hand, our original proof of Theorem \ref{T:FiniteUnram} \emph{does} in principle allow one to obtain explicit estimates on the size of $\gen(D)$ in certain cases. We will only sketch the main ideas here; further details will be available in \cite{CRR3}. For simplicity, suppose that $K$ is a finitely generated field of characteristic 0. Then $K$ can be realized as the function field $k(C)$ of a smooth projective geometrically irreducible curve $C$ over a field $k$ which is a purely transcendental extension of a number field $P$. Let $V_0$ be the set of geometric places of $K$. It is well-known (see, e.g., \cite{Lich}) that the geometric Brauer group $\Br(K)_{\rm ur} = \Br(K)_{V_0}$ fits into an exact sequence
%Let $V_0$ be the set of {\it geometric places} of $K = k(C)$, i.e. the places that are trivial on $k$, or, equivalently, those corresponding to the closed points of $C$. Then, for %each $v \in V_0$, the residue field $K^{(v)}$ is a finite extension of $k$; consequently, we have a residue map as above for each $n > 1$, and hence a residue map $\rho_v \colon %\Br(K) \to \Hom (G^{(v)}, \Q/ \Z)$. The {\it geometric Brauer group} of $K$ is then defined as $\Br(K)_{V_0}$, which, following tradition, we will denote in this paper by %$\Br(K)_{\rm ur}.$ A well-known fact about the geometric Brauer group is that it fits into an exact sequence
\begin{equation}\label{E:GeomExact}
\Br(k) \stackrel{\iota_k}{\longrightarrow} \Br(k(C))_{\rm ur} \stackrel{\omega}{\longrightarrow} H^1 (k,J)/\Phi (C,k),
\end{equation}
where $\iota_k$ is the natural map, $J$ is the Jacobian of $C$, and $\Phi (C,k)$ is a certain finite cyclic subgroup of $H^1(k,J)$ (in fact, if $C(k) \neq \emptyset$, then $\Phi(C,k) = 0$ and (\ref{E:GeomExact}) becomes a split exact sequence). We take our required set $V$ of discrete valuations of $K$ to consist of $V_0$, together with a set ${V}_1$ of extensions to $K$ of an appropriate set of valuations of the field $k$ (if $k$ is a number field, then $V_1$ consists of extensions to $K$ of almost all places of $k$). Then ${}_n\Br(K)_V \subset {}_n\Br(K)_{\rm ur}$ and the proof of Theorem \ref{T:FiniteUnram} reduces to verifying the finiteness of
%Thus, to prove Theorem 1, it suffices to find a set $V$ of discrete valuations of $K$ containing $V_0$, so that ${}_n\Br(K)_V \subset {}_n\Br(K)_{\rm ur}$, and such that
$$\iota_k^{-1}({}_n\Br(K)_V) \ \ \  {\rm and}  \ \ \ \omega({}_n\Br(K)_V).$$
The proof of the finiteness of $\iota_k^{-1}({}_n\Br(K)_V)$ relies on certain properties of our presentation of $K$ as $k(C)$, together with Faddeev's exact sequence (to relate the Brauer group of $k$ to that of the number field $P$) and (ABHN). Our approach for proving the finiteness of $\omega({}_n\Br(K)_V)$ is inspired by the proof of the Weak Mordell-Weil theorem for elliptic curves and involves the analysis of unramified cohomology classes. We should point out that this argument, although with some extra work, also allows one to delete from the constructed set $V$ any finite subset and still retain the finiteness of ${}_n\Br(K)_V$.

%We should point out that, just like the second proof mentioned above, this argument also allows one to delete any {\it finite} set from $V$.

To illustrate things in more concrete terms, we would like to conclude this section by sketching the proof of Theorem \ref{T:FiniteUnram} in the case that $K$ is the function field of an elliptic curve over a number field and $n = 2$ (see \cite[\S 4]{CRR2}).
The set-up that we will consider is as follows. Let $k$ be a number field and $E$ be an elliptic curve over $k$
given by a Weierstrass equation
\begin{equation}\label{E:Elliptic0}
y^2 = f(x), \ \ \text{where} \ \ f(x) = x^3 + \alpha x^2 + \beta x +
\gamma.
\end{equation}
Denote by $\delta \neq 0$ the discriminant of $f$. We will assume that $E$ has $k$-rational 2-torsion, i.e. $f$ has three
(distinct) roots in $k$:
$$
f(x) = (x-a)(x-b)(x-c).
$$
Let
$$
K := k(E) = k(x , y)
$$
be the function field of $E$. Since $E(k) \neq \emptyset$ and $E$ coincides with its Jacobian, the sequence (\ref{E:GeomExact}) yields the exact sequence
\begin{equation}\label{E:EllCurveSeq}
0 \to {}_2\Br(k) \to {}_2\Br(K)_{\rm ur} \stackrel{\omega}{\longrightarrow} {}_2H^1(k, E) \to 0.
\end{equation}
In fact, this sequence is split, with a section to $\omega$ constructed as follows. The Kummer sequence
$$
0 \to E[2] \to E \stackrel{\times 2}{\longrightarrow} E \to 0
$$
yields the exact sequence of cohomology
\begin{equation}\label{E:Kummer}
0 \to E(k)/2 E(k) \to H^1 (k, E[2]) \stackrel{\sigma}{\longrightarrow} {}_2H^1(k, E) \to 0.
\end{equation}
Since $E[2] \simeq \Z/ 2 \Z \times \Z/ 2 \Z$ as Galois modules, we have
$$
H^1(k, E[2]) \simeq k^{\times}/(k^{\times})^2 \times k^{\times}/(k^{\times})^2,
$$
and we define a map
$$
\nu \colon H^1 (k, E[2]) \to {}_2\Br(k(E))_{\rm ur}, \ \ (r, s) \mapsto \left[ \left(\frac{r \ , \ x - b}{K}
\right) \otimes_K \left(\frac{s \ , \ x - c}{K}  \right) \right].
$$
One then checks that $\omega \circ \nu = \sigma$ and $\nu (\ker \sigma) = 0$, which yields the required section
$$
\varepsilon \colon {}_2H^1(k,E) \to {}_2\Br(k(E))_{\rm ur}.
$$
This leads to the following description of the geometric Brauer group, which is in fact valid over any field $k$ of characteristic $\neq 2, 3.$
\begin{thm}\label{T:CGu}{\rm (\cite[Theorem 3.6]{CGu})}
Assume that the elliptic curve $E$
given by $(\ref{E:Elliptic0})$ has $k$-rational 2-torsion, i.e. $$f(x)=(x -
a)(x - b)(x - c) \ \  \text{with} \ \  a, b, c \in k.$$ Then
$$
{}_2\Br(K)_{\rm ur} = {}_2\Br(k) \oplus I,
$$
where ${}_2\Br(k)$ is identified with a subgroup of ${}_2\Br(K)$ via
the canonical map $\Br(k) \to \Br(K)$, and $I \subset
{}_2\Br(K)_{\rm ur}$ is a subgroup such that every element of $I$ is
represented by a bi-quaternion algebra of the form
$$
\left(\frac{r \ , \ x - b}{K}
\right) \otimes_K \left(\frac{s \ , \ x - c}{K}  \right)
$$
for some $r , s \in k^{\times}$.
\end{thm}
As we mentioned above, the required set $V$ will consist of the geometric places $V_0$ together with a set $V_1$ of extensions of almost all places of $k$ to $K$, which is obtained as follows. For $s \in k^{\times}$, we denote by $V^k(s)$ the finite set
$\{ v \in V^k \setminus V_{\infty}^k \: \vert \: v(s) \neq 0 \}.$ Fix
a finite set of valuations $S
\subset V^k$ containing $V_{\infty}^k \cup V^k(2) \cup V^k(\delta)$,
%where for $s \in k^{\times}$ we let $V^k(s)$ denote the finite set
%$\{ v \in V^k \setminus V_{\infty}^k \: \vert \: v(s) \neq 0 \}$,
as
well as all those nonarchimedean $v \in V^k$ for which at least one of
$\alpha, \beta, \gamma$ has a negative value. For a nonarchimedean $v
\in V^k$, let $\tilde{v}$ denote its extension to $F := k(y)$ given
by
\begin{equation}\label{E:Elliptic1}
\tilde{v}(p(y)) = \min_{a_i \neq 0} v(a_i) \ \ \text{for} \ \ p(y) =
a_n y^n + \cdots + a_0 \in k[y], \ \ p \neq 0.
\end{equation}
Now $K$ is a cubic extension of $F$, and one shows that for
$v \in V^k \setminus S$, the valuation $\tilde{v}$ has a {\it
unique} extension to $K$, which we will denote by $w = w(v)$ (see \cite[Lemma 4.5]{CRR2}). We set
$$
V_1 = \{ w(v) \mid v \in V^k \setminus S \}
$$
and let
$$
V = V_0 \cup V_1.
$$

\begin{prop}
The unramified Brauer group ${}_2\Br(K)_V$ is finite.
\end{prop}
\begin{proof} (Sketch)
We consider separately the ramification behavior at places in $V_1$ of the constant and bi-quaternionic parts of elements of ${}_2\Br(K)_{\rm ur}$ in the decomposition given by Theorem \ref{T:CGu}. Let us write $[D] \in {}_2\Br(K)_V$ in the form
$$
[D] = [\Delta \otimes_k K] \otimes_K \left[ \left(\frac{r \ , \
x - b}{K} \right) \otimes_K \left(\frac{s \ , \ x - c}{K}
\right) \right],
$$
where $[\Delta] \in {}_2\Br(k)$ is a quaternion algebra and $r, s \in k^{\times}.$ By using properties of corestriction as well as an explicit description of residue maps in this situation (see \cite[Proposition 4.7 and Lemma 4.8]{CRR2}), one shows that for any $v \in V^k \setminus S$, the quaternion algebra $\Delta$ is unramified at $v$ and also that
\begin{equation}\label{E:EvenVal}
v(r), v(s) \equiv 0 \ (\text{mod} \ 2).
\end{equation}
The finiteness of ${}_2\Br(K)_V$ then follows. Indeed, as we saw in Example 4.1(b), the unramified Brauer group ${}_2\Br(k)_{V^k \setminus S}$ is finite, and hence there are only finitely many possibilities for $[\Delta].$ On the other hand, it is a well-known consequence of the finiteness of the class number and the fact that the group of units is finitely generated that the image under the canonical map $k^{\times} \to k^{\times}/(k^{\times})^2$ of the set
$$
P(k,S) = \{ x \in k^{\times} \: \vert \: v(x) \equiv
0(\mathrm{mod}\: 2) \ \ \text{for all} \ \ v \in V^k \setminus S \}
$$
is finite, which yields the finiteness of the bi-quaternionic part.
\end{proof}

We will now sketch a cohomological proof of (\ref{E:EvenVal}), which is similar to an argument used in the standard proof of the Weak Mordell-Weil Theorem (see, e.g. \cite[Ch. VIII, \S2]{Silv}).
%We should point out that the finiteness of the bi-quaternionic part can also be formulated in cohomological terms, which is actually what is needed for the consideration of the general case.
First, we recall the following definition. Let $v \in V^k \setminus V_{\infty}^k$. We say that $x \in H^1(k, E[2])$ is {\it unramified} at $v$ if
$$
x \in \ker (H^1(k, E[2]) \stackrel{\text{res}_v}{\longrightarrow} H^1(k_v^{\rm ur}, E[2])),
$$
where $k_v^{\rm ur}$ is the maximal unramified extension of the completion $k_v$ and $\text{res}_v$ is the usual restriction map. Furthermore, given a set $U \subset V^k \setminus V_{\infty}^k$, we define the corresponding unramified cohomology group by
$$
H^1 (k, E[2])_U = \bigcap_{v \in U} \ker (H^1(k, E[2]) \stackrel{\text{res}_v}{\longrightarrow} H^1 (k_v^{\rm ur}, E[2])).
$$
%More precisely, for $V' = V^k \setminus S,$ we define the {\it unramified cohomology group}
%$$
%H^1 (k, E[2])_{V'} = \bigcap_{v \in V'} \ker (H^1(k, E[2]) \stackrel{\text{res}_v}{\longrightarrow} H^1 (k_v^{\rm ur}, E[2])),
%$$
%where $k_v^{\rm ur}$ is the maximal unramified extension of the completion $k_v$ and $\text{res}_v$ is the usual restriction map.
Now, one shows that if $x \in I$ lies in  ${}_2\Br(K)_V$, then 
$$
\sigma^{-1}(\omega(x)) \subset H^1(k, E[2])_{V^k \setminus S},
$$
where $\sigma \colon H^1(k, E[2]) \to {}_2H(k, E)$ is the map appearing in (\ref{E:Kummer}). On the other hand, it is well-known that in the canonical isomorphism
\begin{equation}\label{E:UR1}
H^1(k_v^{\rm ur} , \Z/ 2 \Z) \simeq (k_v^{\rm ur})^{\times}/{(k_v^{\rm ur})^{\times}}^{2},
\end{equation}
a coset $a{(k_v^{\rm ur})^{\times}}^{2} \in
(k_v^{\rm ur})^{\times}/{(k_v^{\rm ur})^{\times}}^{2}$ corresponds to the character
$$\chi_a \colon \Ga((k_v^{\rm ur})^{\mathrm{sep}}/k_v^{\rm ur}) \to \Z / 2 \Z$$ with kernel
$\ker \chi_a =
\Ga((k_v^{\rm ur})^{\mathrm{sep}}/k_v^{\rm ur}
(\sqrt{a}))$ (see, e.g., \cite[Proposition 4.3.6 and Corollary 4.3.9]{GiSz}). It follows that if $a
{(k_v^{\rm ur})^{\times}}^2$ corresponds to a cohomology class that is
unramified at $v$, then $\sqrt{a} \in k^{\mathrm{ur}}_v$, and
consequently $v(a) \equiv 0(\mathrm{mod}\: 2)$ (see \cite[Proposition 1.3]{LaDG}). In particular, this, together with the description of the geometric Brauer group given above, shows that if
$$
\left[ \left(\frac{r \ , \ x - b}{K}
\right) \otimes_K \left(\frac{s \ , \ x - c}{K}  \right) \right] \in I
$$
is unramified at a place $w(v) \in V_1$, then (\ref{E:EvenVal}) holds.
%$$
%v(r), v(s) \equiv 0 (\text{mod} \ 2).
%$$

We would also like to observe that the proof sketched above not only gives
the finiteness of ${}_2\Br(K)_V$, but in fact also yields an explicit upper bound on the size of the unramified Brauer group. More precisely, we have

\begin{thm}\label{T:Elliptic1}
For \emph{any} finite set $S$ as above, the unramified Brauer group ${}_2\Br(K)_V$ is finite of order
dividing
$$
2^{\vert S \vert - t} \cdot \vert {}_2 \mathrm{Cl}_S(k) \vert^2
\cdot \vert U_S(k)/U_S(k)^2 \vert^2,
$$
where $t = c + 1$ and $c$ is the number of complex places of
$k$, and $\mathrm{Cl}_S(k)$ and $U_S(k)$ are the class group and the
group of units of the ring of $S$-integers $\mathcal{O}_k(S)$,
respectively.
\end{thm}

\vskip2mm

\noindent {\bf Example 4.7.} Consider the elliptic curve $E$ over
$\Q$ given by $y^2 = x^3 - x$. We have $\delta = 4$, so $S =
\{\infty , 2 \}$. Furthermore,
$$
\vert S \vert - t = 1, \ \ \mathrm{Cl}_S(\Q) = 1 \ \ \text{and} \ \
U_S(\Q) = \{\pm 1\} \times \Z.
$$
So, by Theorem \ref{T:Elliptic1} and (\ref{E:SizeImageGenus2})
%So, by Theorem \ref{T:Elliptic1}, for $K = \Q(E)$ and the set $V$
%constructed above, the group ${}_2\Br(K)_V$ has order dividing $2
%\cdot 4^2 = 32$. Combining this with Theorem \ref{P:R2}, we obtain
%that for any quaternion division algebra $D$ over $K$,
we have $\vert \gen(D) \vert \leqslant 2 \cdot 4^2 = 32$ for any quaternion division algebra $D$ over $K = \Q(E).$

%Finally, for the proof of Theorem \ref{T:StabThm}, we view
%the field $K = k(x)$ as the function field of $\mathbb{P}^1_k$, and take $V$ to be the set of places corresponding to the closed points of $\mathbb{P}^1_k$ (in other words, $V$ consists of the discrete valuations of $K$ that are trivial on $k$). The key fact that is needed, and which follows from Faddeev's exact sequence, is that
%for any integer $n > 1$ prime to ${\rm char}~k,$ the unramified Brauer group ${}_n\Br(K)_V$ coincides with ${}_n\Br(k)$ (cf. \cite{GiSz}...). Note that for curves of higher genus, the unramified Brauer group is generally much bigger than the Brauer group of the base field (e.g., see \cite{CGu} for computations of the unramified Brauer group of the function field of an elliptic curve). Nevertheless, using techniques similar to the ones discussed in this section, we have been able to obtain some estimates on the size of the genus for division algebras over function fields of smooth projective curves in many cases (see \cite[Theorem 3.1]{CRR2}).

%Nevertheless, we have been able to obtain in \cite[Theorem 3.1]{CRR1} some estimates on the size of the genus for division algebras over function fields of quite general smooth projective curves.

%, one shows that the $n$-torsion of the \'etale Brauer
%to establish the finiteness of the \'etale Brauer group in this setting
%two major results: Deligne's finiteness theorems for \'etale cohomology and results of Gabber that imply that the \'etale Brauer group coincides with the unramified Brauer group.

\section{Some other notions of the genus}\label{S:OtherGenus}

In this section, we would like to mention several variants of the notion of the genus of a division algebra that come up quite naturally and have proven to be useful.

First, we have the {\it local genus}, which already appeared implicitly in the proof of Lemma \ref{L:=}. Let $K$ be a field and $V$ a set of discrete valuations of $K$. For a central division algebra $D$ of degree $n$ over $K$, we define the {\it local genus} $\gen_V(D)$ of $D$ with respect to $V$ as the collection of classes $[D'] \in \Br(K)$,
with $D'$ a central division algebra of degree $n$ over $K$, such that for every $v \in V$ the following holds: if we write
%with $D'$  a central division algebra of degree $n$ over $K$, such that if we write
$$
D \otimes_K K_v = M_{\ell}(\mathcal{D}) \ \  \text{and} \ \  D' \otimes_K K_v =
M_{\ell'}(\mathcal{D}'),
$$
where $\mathcal{D}$ and $\mathcal{D}'$ are central division algebras over the completion $K_v$, then $\ell = \ell'$ and $\mathcal{D}$ and $\mathcal{D'}$ have the same maximal separable subfields. Now, Lemma 2.3 of \cite{RR} (see also \cite[Lemma 3.1]{GS}) shows that for any set $V$ of discrete valuations of $K$, we have
$$
\gen (D) \subset \gen_V(D).
$$
%if $D'$ is a division algebra such that $[D'] \in \gen(D)$, then we automatically have $[D'] \in \gen_V (D)$ for any set $V$ of discrete valuations of $K$ (informally, this means that the property of having the same maximal subfields is preserved under passage to completions).
Next, suppose that $V$ is a set of discrete valuations satisfying condition (I) and (II) introduced in \S \ref{S:Ramification} for some integer $n > 1$. Then the argument that we sketched in \S \ref{S:Ramification} shows that if
%suppose that a set of discrete valuations $V$ satisfies conditions (I)-(III) introduced in \S 4 for some integer $n > 1$ (you give only two conditions - in the paper we have 3; add the 3rd condition on p. 10 to be on the safe side; yes, it holds automatically if the field is finitely generated, and maybe can even be avoided since our algebra has only finitely many structure constants. Or maybe not add (III)? )  Then the argument we sketched in \S 4 show that
{\it $n$ is prime to $\mathrm{char} \: K$ and the unramified Brauer group ${}_n\Br(K)_V$ is finite, then $\gen_V(D)$ is finite for any central division algebra $D$ of degree $n$ over $K$}. Indeed, the proof of Lemma \ref{L:=} yields the fact that the characters $\chi_v = \rho_v([D])$ and $\chi'_v = \rho_v([D'])$  for $v \in V$ have the same kernel assuming only that $[D'] \in \gen_V(D)$. This remark eventually leads to the bound (\ref{E:SizeImageGenus2}), with $\gen_V(D)$ in place of $\gen(D)$.
%suppose that $K$ is a finitely generated field and $V$ a set of discrete valuations of $K$ satisfying conditions (I) and (II) introduced in \S\ref{S:Ramification} for some integer $n > 1.$ If $n$ is prime to $\text{char} \ K$ and the unramified Brauer group ${}_n\Br(K)_V$ is finite, then $\gen_V(D)$ is finite for any central division algebra $D$ of degree $n$ over $K$. Indeed, by the same argument as given in the proof Lemma \ref{L:=}, we see that for any $[D'] \in \gen_V(D)$ and $v \in V$, the characters $\chi_v = \rho_v ([D])$ and $\chi'_v = \rho_v ([D'])$ have the same kernel. This leads to the bound (\ref{E:SizeImageGenus2}), with $\gen_V(D)$ in place of $\gen(D).$

Next, as we already mentioned in Remark 3.7(b), D.~Krashen and K.~McKinnie \cite{KM} have studied division algebras having the same finite-dimensional {\it splitting} fields. For this purpose, given a finite-dimensional central division algebra $D$ over a field $K$, one defines $\gen'(D)$ as the collection of classes $[D'] \in \Br(K)$ with the property that a finite field extension $L/K$ splits $D$ if and only if it splits $D'$ (the notation in \cite{KM} is slightly different). Note that we clearly have $$\gen'(D) \subset \gen (D).$$ The main result of \cite{KM} is that if $p$ is a prime different from $\text{char}~K$ and $$\vert \gen'(\Delta) \vert < \infty$$ for all $[\Delta] \in {}_p\Br(K)$, then
$$
\vert \gen'(D) \vert < \infty
$$
for all $[D] \in {}_p\Br(K(t))$ (a similar result for $\gen(D)$ was obtained in \cite{CRR2}, but technically neither result is a consequence of the other).

The third notion that we would like to discuss briefly is the so-called {\it one-sided} (or {\it asymmetric}) {\it genus} that was introduced in \cite{Common}. Following \cite{Common}, given two central division algebras $D$ and $D'$ of the same degree over a field $K$, we will write
%Before formulating the definition, it is first useful to have the following notation. Let $D$ and $D'$ be central division algebras of the same degree over a field $K$. We will write
$$
D \leq D'
$$
if any maximal subfield $P/K$ of $D$ admits a $K$-embedding $P \hookrightarrow D'.$ For a division algebra $D$ of degree $n$ over $K$, we define the one-sided genus
$\gen^1(D)$ to be the collection of classes $[D'] \in \Br(K)$, where $D'$ is a central division algebra of degree $n$ such that $D \leq D'.$ We refer the reader to \cite{Common} for a detailed treatment of this notion, and only highlight here the difference in the ramification properties arising in the analysis of the two-sided and one-sided versions of the genus.
%A detailed treatment of this and related notions will be given in \cite{Common}; here we would only like to make some remarks concerning ramification properties.

By Lemma 4.2, if $[D'] \in \gen(D)$, then for a discrete valuation $v$ of $K$, the algebra $D$ ramifies at $v$ if and only if $D'$ does (assuming that the residue map $\rho_v$ is defined).
%Recall that one consequence of Lemma \ref{L:=} is that if $D$ and $D'$ are division algebras over $K$ such that $[D'] \in \gen(D)$ and $v$ is a discrete valuation of $K$ for which the residue map $\rho_v$ is defined, then both algebras are simultaneously either ramified or unramified at $v.$
For the one-sided genus, the situation is more complicated. If $K$ is a number field, then, as noted prior to sketching the proof 
of Proposition \ref{P:NF}, the relation $D \leq D'$ implies $R(D') \subset R(D)$ for the corresponding sets of ramification places.
%whenever $D \leq D'$, the set of ramification places $R(D')$ of $D'$ is contained in $R(D).$
On the other hand, let $K = \R((x))$ with the standard
discrete valuation $v$, and consider the quaternion division $K$-algebras
$$
D_1 = \left(  \frac{-1 , -1}{K} \right) \ \ , \ \ D_2 = \left(
\frac{-1 , x}{K} \right).
$$
Then $D_1$ is
unramified at $v$ with residue algebra isomorphic to the usual
algebra of Hamiltonian quaternions $\mathbb{H}$ over $\R$, while $D_2$
is ramified at $v$. Note that any quadratic subfield $L$ of $D_1$
must be unramified, and since the residue field $\overline{K}_v =
\R$ has $\C$ as its only nontrivial finite extension, we conclude
that $L$ is isomorphic to $K(i) = \C((x))$ (where $i^2 = -1$). Since
$\C((x))$ is contained in $D_2$, we see that any maximal subfield of
$D_1$ can be embedded into $D_2$. This means that $D_1 \leq D_2$, hence $[D_2] \in \gen^1(D_1).$ Nevertheless, $D_1$ is unramified
at $v$ while $D_2$ is ramified.

The results of \cite{Common} show that
%As shown in \cite{Common},
this construction provides essentially the main instance where such ramification behavior appears. To give precise statements, we need to introduce some terminology and fix notations. In \cite{Common}, a finite-dimensional central division algebra $D$ over a field $K$ is called {\it varied}
%We will say that a finite-dimensional central division algebra $D$ over a field $K$ is {\it varied}
if there is no nontrivial cyclic
extension $P/K$ contained isomorphically in every maximal subfield of $D$ (note that it suffices to check this property for $P/K$ of prime degree). For example, it is known that if $K$ is a field that is finitely generated over a global field, then any central division algebra $D$ over $K$ is varied (see \cite[Theorem 1]{Common}).
%Furthermore, keeping the notation introduced in the proof of Lemma \ref{L:=},
For a central division algebra $D$ over a field $K$ that is complete with respect to a discrete valuation, we will denote by $\overline{D}$ the residue algebra and by ${\overline{E}}_D$ the center of $\overline{D}$; we will also set $E_D$ to be the unique unramified subfield of $D$ with residue field $\overline{E}_D$ (in \cite{Common}, the latter is referred to as the {\it ramification field} of $D$).

\begin{prop}\label{P:Common1}{\rm (\cite{Common}, Proposition 4)}
Let $K$ be a field that is complete with respect to a discrete valuation $v$, and $D$, $D'$ be central division algebras over $K$ of the same degree such that $D \leq D'.$ If $\overline{D}/{\overline{E}}_D$ is varied, then $E_D = E_{D'}.$
\end{prop}

To relate this result to the above discussion of ramification, suppose that there exists a residue map
$$
\rho_v \colon {}_n\Br(K) \to \Hom (\mathcal{G}^{(v)}, \Z / n \Z),
$$
where $n = \text{deg} \ D.$ As we mentioned earlier, it is well-known that if $\chi_v = \rho_v ([D]),$ then $\overline{E}_D$ is precisely the subfield of the separable closure of $\overline{K}_v$ corresponding to $\ker \chi_v.$ Thus, Proposition \ref{P:Common1} asserts that
under the assumption that $\overline{D}/\overline{E}_D$ is varied, the conclusion of Lemma \ref{L:=} holds already if $D \leq D'$, and in particular
%Proposition \ref{P:Common1} asserts that
$D$ and $D'$ are simultaneously either ramified or unramified at $v$. Of course, the above example of quaternion algebras over $K = \R((x))$ shows that this assumption cannot be omitted --- notice that $E_{D_1} = \R((x))$ while $E_{D_2} = \C((x))$. On the other hand, it turns out that all situations where a division algebra is not varied are in some sense of this nature.

\begin{prop}\label{P:Common2}{\rm (\cite{Common}, \S3)}

\noindent (a) \parbox[t]{16cm}{Suppose $D$ is a finite-dimensional central division algebra over a field $K$ that is not varied. Then $K$ is a Pythagorean field and
$$
D = \left(\frac{-1, -1}{K} \right) \otimes_K D'
$$
for some finite-dimensional central division $K$-algebra $D'.$}
%Then says that a non-varied division algebra has the
%form $D =  \mathbb H \otimes D',$ for some $D'/F$ and $F$
%Pythagorean.

\vskip1mm

\noindent (b) \parbox[t]{16cm}{Suppose $K$ is a field that is complete with respect to a discrete valuation and let $D$, $D'$ be central division algebras over $K$ of the same degree such that $D \leq D'$. Then $E_DE_{D'} \subseteq E_{D}(\sqrt{-1})$. If $E_D \subsetneq E_{D'} = E_{D}(\sqrt{-1})$, then the degree $[E_D : K]$ is odd and $K = F((x))$ for a Pythagorean field $F$.}
%$L_D = L_E(\sqrt{-1}) \not= L_E$, then
%$L_E/F$ has odd degree $F = K((\pi))$ for a Pythagorean field $K$.
\end{prop}
\noindent (Recall that a field $F$ is said to be {\it Pythagorean} if every sum of two squares in $F$ is a square.)
%In this section, we briefly discuss two other notions of the genus of a division algebra that come up quite naturally --- the {\it local genus} and the {\it asymmetric genus.}

\section{The genus of an algebraic group}\label{S:GenAlgGp}

We would like to conclude this article with a brief overview of ongoing work whose goal is to extend the techniques and results developed in the context of division algebras (which we have outlined in  \S\S \ref{S:Genus}-\ref{S:OtherGenus}) to absolutely almost simple algebraic groups of all types. In this case, the notion of division algebras having the same maximal subfields is replaced with the notion of algebraic groups having the same maximal tori.  More precisely, let $G_1$ and $G_2$ be absolutely almost simple algebraic groups defined over a field $K$. We say that $G_1$ and $G_2$ have the {\it same} $K$-{\it isomorphism} (resp., $K$-{\it isogeny}) {\it classes of maximal $K$-tori} if every maximal $K$-torus $T_1$ of $G_1$ is $K$-isomorphic (resp., $K$-isogenous) to some maximal $K$-torus $T_2$ of $G_2$, and vice versa. Furthermore, let
$G$ be an algebraic $K$-group and $K^{\mathrm{sep}}$ a separable closure of $K$.  We recall that an algebraic $K$-group $G'$ is called a $K$-form (or, more precisely, a $K^{\mathrm{sep}}/K$-form) of $G$ if $G$ and $G'$ become isomorphic over $K^{\mathrm{sep}}$ (see, e.g.,  \cite[Ch. III, \S 1]{Serre} or \cite[Ch. II, \S 2]{Pl-R}). For example, for any central division algebra $D$ of degree $n$ over $K$, there exists a $K^{\mathrm{sep}}$-isomorphism $D \otimes_K K^{\mathrm{sep}} \simeq M_n(K^{\mathrm{sep}})$, which means that the algebraic $K$-group $G = \mathrm{SL}_{1 , D}$ associated with the group of elements in $D$ having reduced norm 1 is a $K$-form of $\mathrm{SL}_n$.

%We would like to conclude by giving in this section a short survey of some work in progress whose goal is to extend the results and approaches of the previous sections to the context of arbitrary algebraic groups. To begin with, a natural generalization of division algebras with the same maximal subfields are algebraic groups with the same maximal tori. More precisely, suppose that $G_1$ and $G_2$ are algebraic groups defined over a field $K$. We say that $G_1$ and $G_2$ have the {\it same $K$-isomorphism} (resp. {\it same $K$-isogeny}) {\it classes of maximal $K$-tori} if every maximal $K$-torus $T_1$ of $G_1$ is $K$-isomorphic (resp. $K$-isogenous) to some maximal $K$-torus $T_2$ of $G_2$, and vice versa. We introduce the following definition.

\vskip2mm

\noindent {\bf Definition 6.1.} Let $G$ be an absolutely almost simple algebraic
group over a field $K$. The $(K$-$)${\it genus} $\gen_K(G)$ (or
simply $\gen(G)$ if there is no risk for confusion) of $G$ is
the set of $K$-isomorphism classes of $K$-forms $G'$ of $G$
that have the same $K$-isomorphism classes of maximal $K$-tori as
$G$.

\addtocounter{thm}{1}

\vskip2mm

We should point out that for a finite-dimensional central division $K$-algebra $D$, only maximal {\it separable} subfields of $D$ give rise to maximal $K$-tori of the corresponding group $G = \mathrm{SL}_{1 , D}$. So, in hindsight, to make the definition of $\gen(D)$ consistent with that of $\gen(G)$, one should probably reformulate the former in terms of maximal {\it separable} subfields. This change would not affect the results that were discussed in \S\S \ref{S:Genus}-\ref{S:Ramification} as these dealt
with the case where the degree of the algebra is prime to $\text{char}~K$, but its potential impact on the general case has not yet been investigated. We also observe that while we will be interested primarily in simple algebraic groups with the same {\it isomorphism} classes of maximal tori, the analysis of weakly commensurable Zariski-dense subgroups, which is related to geometric applications (see \cite{PR1}), sometimes requires one to consider simple groups with the same {\it isogeny} classes of maximal tori.

As in the case of division algebras, the focus of our current work is on the following two questions:

%We should point out that for $G = \mathrm{SL}_{1 ,
%D}$, where $D$ is a finite-dimensional central division $K$-algebra,
%only maximal {\it separable} subfields of $D$ give rise to maximal
%$K$-tori of $G$. So, to make our definition of $\gen(D)$ consistent with the above definition, the former should probably be formulated not in terms of all maximal subfields, but %only in terms of the {\it separable} ones. Since we have only analyzed division algebras of degree prime to $\text{char} \ K$, such a modification would not affect the results of the previous sections, but the general case has not been investigated. Furthermore, while our definition of $\gen(G)$ is formulated in terms of groups having the same isomorphism classes of maximal tori as $G$, in practice, it is often useful to consider groups having the same isogeny classes of maximal tori as well.

%As in the case of division algebras, the following two questions are currently of primary interest to us.

\medskip

\noindent \parbox[t]{17cm}{{\bf Question ${\bf 1'}$.} {\it When does $\gen_K(G)$ reduce to a single
element?} \vskip1mm (This means that among $K$-groups of the same
type, $G$ is defined up to $K$-isomorphism by the isomorphism classes
of its maximal $K$-tori.)}

\medskip

\noindent \parbox[t]{16cm}{{\bf Question ${\bf 2'}$.} {\it When is $\gen_K(G)$ finite?}}

\medskip

At this point, only the case of absolutely almost simple algebraic groups over number fields has been considered in full.

%At this point, the complete answer in the context of arbitrary groups is known only when $K$ is a number field.

%At this point, definitive results on these problems are fairly limited. In the case that $K$ is a number field, the following result was obtained in \cite{PR1}.

\begin{thm}{\rm (cf. \cite[Theorem 7.5]{PR1})}\label{T:WC13}
{\rm (1)} Let $G_1$ and $G_2$ be connected absolutely almost simple
algebraic groups defined over a number field $K$, and let $L_i$ be
the smallest Galois extension of $K$ over which $G_i$ becomes an
inner form of a split group. If $G_1$ and $G_2$ have the same
$K$-isogeny classes of maximal $K$-tori then either $G_1$ and $G_2$
are of the same Killing-Cartan type, or one of them is of type
$\textsf{B}_{n}$ and the other is of type $\textsf{C}_{n}$ $(n \geq 3)$, and
moreover, $L_1 = L_2$.

\smallskip

\noindent {\rm (2)} \parbox[t]{16.5cm}{Fix an absolutely almost
simple $K$-group $G$. Then the set of isomorphism classes of all
absolutely almost simple $K$-groups $G'$ having the same $K$-isogeny
classes of maximal $K$-tori as $G$, is finite.}

\smallskip

\noindent {\rm (3)} \parbox[t]{16.5cm}{Fix an absolutely almost
simple simply connected $K$-group $G$ whose Killing-Cartan type is
different from $\textsf{A}_{n}$, $\textsf{D}_{2n+1}$ $(n > 1)$ or
$\textsf{E}_6$. Then any $K$-form $G'$ of $G$ (in other words, any
absolutely almost simple simply connected $K$-group $G'$ of the
\emph{same} type as $G$) that has the same $K$-isogeny classes of
maximal $K$-tori as $G$, is isomorphic to $G$.}
\end{thm}

\noindent Regarding the types excluded in (3), the construction in \cite[\S 9]{PR1} shows that they are honest exceptions, i.e., for each of those types
one can construct non-isomorphic absolutely almost simple simply
connected $K$-groups $G_1$ and $G_2$ of this type over a number
field $K$ that have the same isomorphism classes of maximal
$K$-tori. The case where $G_1$ and $G_2$ are of types $\textsf{B}_{n}$ and
$\textsf{C}_{n}$, respectively, has been analyzed fully in \cite{GarR}.

We now observe that the investigation of $\gen(G)$ presents additional challenges even for groups of the form $G = \mathrm{SL}_{m , D}$, where $D$ is a central division algebra of degree $n$ (we recall that $G$
is a simply connected inner form of type  $\textsf{A}_{\ell}$ with $\ell = mn - 1$, and that all inner forms of this type are obtained in this fashion --- see \cite[Proposition 2.17]{Pl-R}). The reason is that while every maximal $K$-torus of such a group $G$ is a norm one torus
$$
\mathrm{R}^{(1)}_{E/K}(G_m) = \mathrm{R}_{E/K}(G_m) \cap G
$$
for some \'etale subalgebra $E$ of $M_m(D)$, where $G_m$ denotes the 1-dimensional split torus and $\mathrm{R}_{E/K}$ the functor of restriction of scalars, the fact that two such tori  $\mathrm{R}^{(1)}_{E_1/K}(G_m)$ and $\mathrm{R}^{(1)}_{E_2/K}(G_m)$ are $K$-isomorphic as algebraic groups does not in general imply that the algebras $E_1$ and $E_2$ are isomorphic, even when these algebras are field extensions of $K$. This makes it rather difficult to relate $\gen(G)$ and $\gen(D)$ and apply the results of \S\S \ref{S:Genus}-\ref{S:Ramification} directly.
%hence to apply the results of \S\S directly.
Nevertheless, with some extra work, one can prove the following theorem, which parallels Theorems \ref{T:StabThm} and \ref{T:Finite1} for division algebras.

%Over finitely generated fields, we have the following statement, which parallels our Theorems \ref{T:StabThm} and \ref{T:Finite1} for division algebras.

\begin{thm}{\rm (cf. \cite[Theorem 5.3]{CRR2}\label{T:GenAlgGp1})}
\noindent {\rm (a)} {Let $D$ be a central division algebra of exponent 2 over $K = k (x_1, \dots, x_r)$, where $k$ is either a number field or a finite field of characteristic $\neq 2$. Then for any $m \geq 1$, the genus of $G = \mathrm{SL}_{m, D}$ reduces to a single element.}
%one shows that given a
%central division algebra $D$ of exponent 2 over the field of rational functions
%$K = k(x_1, \ldots , x_r)$, where $k$ is either a number field or a finite field
%of characteristic $\neq 2$, the genus of $G = \mathrm{SL}_{m , D}$ for any $m \geqslant 1$
%reduces to a single element

\vskip1mm

\noindent {\rm (b)} \parbox[t]{16.5cm}{Let $G$ be an absolutely almost simple simply connected algebraic
group of inner type $\textsf{A}_{\ell}$ over a finitely generated field $K$ whose characteristic is either zero or does not divide $\ell + 1$. Then $\gen(G)$ is finite.}

\end{thm}

\noindent The additional input that is needed to prove Theorem \ref{T:GenAlgGp1} are so-called {\it generic tori}. Since these are becoming increasingly useful in a variety of contexts, we will quickly recall here some relevant definitions and results.
%additional technique that is needed to prove Theorem 6.3. is {\it generic tori}. Since this technique is becoming increasingly useful, we quickly recall here some relevant definitions and results.

Let $G$ be a semisimple algebraic group over a field $K$. Fix a maximal $K$-torus $T$ of $G$ and let $\Phi = \Phi(G , T)$  denote the corresponding root system. Furthermore, let $K_T$ be the minimal splitting field of $T$ and $\Theta_T = \mathrm{Gal}(K_T/K)$ be its Galois group. Then the natural action of $\Theta_T$ on the character group $X(T)$ gives rise to an injective group homomorphism
$$
\theta_T \colon \Theta_T \rightarrow \mathrm{Aut}(\Phi),
$$
and we say that $T$ is {\it generic} over $K$ if the image of $\theta_T$ contains the Weyl group $W(\Phi) = W(G , T)$. For example, for $G = \mathrm{SL}_{m , D}$ as above, a maximal $K$-torus $T = \mathrm{R}^{(1)}_{E/K}(G_m)$ is generic if and only if $E$ is a (separable) field extension of $K$ of degree $mn$ and the Galois group of its normal closure is the symmetric group $S_{mn}$.

The following result shows that when $K$ is finitely generated, one can always find a generic $K$-torus with prescribed local properties.

\begin{prop}\label{P:generic} {\rm (cf. \cite[Corollary 3.2]{PR-Fields})}
Let $G$ be an absolutely almost simple algebraic group over a finitely
generated field $K$. Given a discrete valuation $v$ of $K$ and a
maximal $K_v$-torus $T_v$ of $G$, there exists a maximal $K$-torus
$T$ of $G$ which is generic over $K$ and is conjugate to $T_v$ by an
element of $G(K_v)$.
\end{prop}

The second result that we would like to mention is a rigidity property for isomorphisms between generic tori. More precisely, let $G_1$ and $G_2$ be two simply connected inner $K$-forms of type $\textsf{A}_{\ell}$, and let $T_i$ be a generic maximal $K$-torus of $G_i$ for $i = 1, 2$. Then any $K$-isomorphism $\varphi \colon T_1 \to T_2$ extends to a $K^{\mathrm{sep}}$-isomorphism $\tilde{\varphi} \colon G_1 \to G_2$. (This is a consequence of so-called {\it Isogeny Theorem}, see \cite[Theorem 4.2 and Remark 4.4]{PR1} and \cite[Theorem 9.8]{PR-Gen}.) It follows that for {\it generic} $T_i = R^{(1)}_{E_i/K}(G_m)$  as above, the existence of a $K$-isomorphism of tori $T_1 \to T_2$ {\it does imply} the existence of a $K$-isomorphism of algebras $E_1 \to E_2$. We are now in a position to present

\vskip2mm

\noindent {\it Sketch of the proof of Theorem \ref{T:GenAlgGp1}.} Assume that $K$ is a finitely generated field and let $G' \in \gen_K(G)$. One first shows that $G'$ is an inner form over $K$, and in fact $G' = \mathrm{SL}_{m , D'}$ for a central division $K$-algebra $D'$ of degree $n$. Using the results on generic tori outlined above, one proves that for any discrete valuation $v$ of $K$, the algebras $D \otimes_K K_v$ and $D' \otimes_K K_v$ have the same isomorphism classes of maximal \'etale subalgebras. This means that for any set $V$ of discrete valuations of $K$, the class $[D']$ lies in the local genus $\gen_V(D)$. On the other hand, according to Theorem \ref{T:FiniteUnram}, for a finitely generated $K$ of characteristic prime to $n$, there exists a set $V$ of discrete valuations of $K$ satisfying conditions (I) and (II)
%(well, you also use (I) and (II) to label the questions we are considering, so maybe we should rename the properties (A) and (B) - we need to do this consistently)
and such that ${}_n\Br(K)_V$ is finite. As we pointed out in \S \ref{S:OtherGenus}, this yields the finiteness of the local genus $\gen_V(D)$, which completes the proof of part (b) of Theorem \ref{T:GenAlgGp1}. To prove part (a), we write $K = \ell(x_1)$ where $\ell = k(x_2, \ldots , x_r)$, and let $V$ be the set of discrete valuations of $K$ that are trivial on $\ell$. Then as in the proof of Theorem \ref{T:StabThm}, the fact that $[D'] \in \gen_V(D)$ in conjunction with Faddeev's sequence implies that
$$
[D] = [D'] [\Delta \otimes_{\ell} K] \ \ \ \text{in} \ \ \Br(K)
$$
for some central division algebra $\Delta$ over $\ell$. Finally, to prove that $\Delta$ is trivial, we pick a place $v_0 \in V$ of degree one so that both $D$  and $D'$ are unramified at $v_0$, and write
$$
D \otimes_K K_{v_0} = M_s(\mathcal{D}) \  \ \text{and} \ \ D' \otimes_K K_{v_0} = M_{s'}(\mathcal{D}'),
$$
with $\mathcal{D} , \mathcal{D}'$ central division $K_{v_0}$-algebras. Then $s = s'$ and $\mathcal{D}$ and $\mathcal{D}'$ have the same maximal (separable) subfields. It follows that the residue algebras $\overline{\mathcal{D}}$ and $\overline{\mathcal{D}'}$ are central division $\ell$-algebras having the same maximal subfields. Since $\overline{\mathcal{D}}$ has exponent 2, by Theorem \ref{T:StabThm} it has trivial genus. It follows that $[\overline{\mathcal{D}}] = [\overline{\mathcal{D}'}]$, and hence $[\Delta]$ is trivial. Thus, $D \simeq D'$, and consequently $G \simeq G'$. \hfill $\Box$

\vskip2mm

\noindent {\bf Remark 6.5.} The nature of the argument that we have just sketched suggests that it makes sense
to consider an alternative definition of $\gen_K(G)$ over a finitely generated field $K$ given in terms of generic maximal $K$-tori.
%The argument that we have just sketched suggests that a possible alternative definition of $\gen_K(G)$ (at least over finitely generated fields) can be given just in terms of {\it generic} maximal $K$-tori.

\vskip2mm

Building on the finiteness result for the genus of an inner form of type $\textsf{A}_{\ell}$ over a finitely generated field (Theorem \ref{T:GenAlgGp1}(b)), it is natural to propose the following.

%In view of the finiteness statement given in Theorem \ref{T:GenAlgGp1}(b), it seems natural to propose the following conjecture.

\vskip2mm

\noindent {\bf Conjecture 6.6.} {\it Let $G$ be an absolutely almost simple simply connected algebraic
group over a finitely generated field $K$ of characteristic 0 or of \emph{good}\footnotemark characteristic relative to $G$.
Then $\gen_K(G)$ is finite.}

\footnotetext{For each type, the following characteristics are
defined to be {\it bad}: type $\textsf{A}_{\ell}$ - all primes
dividing $(\ell + 1)$, and also $p = 2$ for outer forms; types
$\textsf{B}_{\ell}$, $\textsf{C}_{\ell}$, $\textsf{D}_{\ell}$ - $p =
2$, and also $p = 3$ for $^{3,6}\!\textsf{D}_4$; for type $
\textsf{E}_6$ - $p = 2, 3, 5$; for types $\textsf{E}_7$,
$\textsf{E}_8$ - $p = 2, 3, 5, 7$; for types $\textsf{F}_4$,
$\textsf{G}_2$ - $p = 2, 3$. All other characteristics for a given
type are {\it good}.}
\vskip2mm

\addtocounter{thm}{2}

The proof of Theorem \ref{T:GenAlgGp1} indicates that to approach this conjecture, one
%for a proof of this conjecture one
needs to extend the techniques based on ramification and the analysis of unramified division algebras to absolutely almost simple groups of all types. An adequate replacement of the notion of an unramified central division algebra is the notion of a group with {\it good reduction}. Suppose that $G$ is an absolutely almost simple algebraic group over a field $K$.
%So, let $G$ be an absolutely almost simple algebraic $K$-group.
One says that $G$ has good reduction at a discrete valuation $v$ of $K$ if there exists a reductive group scheme\footnotemark $\mathscr{G}$ over the valuation ring $\mathcal{O}_v \subset K_v$  whose generic fiber $\mathscr{G} \otimes_{\mathcal{O}_v} K_v$ is isomorphic to $G \otimes_K K_v$. We then let $\uG^{(v)}$ denote the reduction $\mathscr{G} \otimes_{\mathcal{O}_v} \overline{K}_v$. The following result extends Lemma \ref{L:=} to simple algebraic groups of all types.

%As we saw above, the proof of Theorem \ref{T:GenAlgGp1} implicitly depends on the analysis of the ramification properties of the division algebra $D$. For general algebraic groups, an adequate replacement of the notion of an
%unramified algebra is the notion of a group with
%\emph{good reduction}. Given a discrete valuation $v$ of $K$,
%we let $\mathcal{O}_v$ denote the valuating ring in the completion $K_v$, with
%we let $K_v$ denote the corresponding completion with valuation ring
%$\mathcal{O}_v$,
%valuation ideal $\mathfrak{p}_v$ and residue
%field $\overline{K}_v = \mathcal{O}_v/\mathfrak{p}_v$. One says that
%an algebraic $K$-group $G$ has {\it good reduction at $v$} if there
%exists a reductive group scheme\footnotemark $\mathscr{G}$ over $\mathcal{O}_v$ with
%generic fiber $G \otimes_K K_v$. The reduction $\mathscr{G}
%\otimes_{\mathcal{O}_v} \overline{K}_v$ will then be denoted
%$\uG^{(v)}$. We have been able to establish the following analogue of Lemma \ref{L:=}

\footnotetext{Let $R$ be a commutative ring and $S = \text{Spec}~R.$ Recall that a {\it reductive $R$-group scheme} is a smooth affine group scheme $G \to S$ such that the geometric fibers $G_{\overline{s}}$ are connected reductive algebraic groups (see \cite[Exp. XIX, Definition 2.7]{SGA3} or \cite[Definition 3.1.1]{Conrad1}).}

\begin{thm}\label{T:Main1} {\rm (\cite{CRR4})}
Let $G$ be an absolutely almost simple simply connected group over a
field $K$, and let $v$ be a discrete valuation of $K$. Assume that
the residue field $\overline{K}_v$ is finitely generated and that
$G$ has good reduction at $v$. Then any $G' \in \gen_K(G)$
also has good reduction at $v$. Furthermore, the reduction
${\uG'}^{(v)}$ lies in the genus $\gen_{\overline{K}_v}(\uG^{(v)})$.
\end{thm}

(We should point out that the proof of this result again makes use of generic tori.)

\vskip2mm

Assume now that the field $K$ is equipped with a set $V$ of discrete
valuations that satisfies the following two conditions:

\medskip

(I) for any $a \in K^{\times}$, the set $V(a) := \{ v \in V \ \vert
\ v(a) \neq 0 \}$ is finite;

\medskip

(III) for any $v \in V$, the residue field $\overline{K}_v$ is
finitely generated.

\medskip

\begin{cor}\label{C:Main1}
If $K$ and $V$ satisfy conditions {\rm (I)} and {\rm (III)}, then for any
absolutely almost simple simply connected algebraic $K$-group $G$,
there exists a finite subset $V_0 \subset V$ (depending on $G$) such
that every $G' \in \gen_K(G)$ has good reduction at all $v \in
V \setminus V_0$.
\end{cor}

It follows from Corollary \ref{C:Main1} that in order to prove Conjecture 6.6, it would
suffice to show that every finitely generated field $K$ can be equipped with
a set $V$ of discrete valuations satisfying conditions (I) and (III) and having the following property:

\medskip

\noindent ($\Phi$) \, \parbox[t]{16cm}{For any absolutely almost
simple algebraic $K$-group $G$ such that $\mathrm{char} \: K$ is
good for $G$ and any finite subset $V_0 \subset V$, the set of
$K$-isomorphism classes of (inner) $K$-forms $G'$ having good reduction
at all $v \in V \setminus V_0$, is finite.}

\medskip

\noindent (Obviously, in this formulation one can assume $G$ to be quasi-split over $K$.)
One expects that {\it divisorial} sets of valuations that appeared in the discussion of Theorem \ref{T:FiniteUnram} (i.e. valuations of $K$ arising from the prime divisors on an appropriate arithmetic scheme $X$ with function field $K$) will also work for general algebraic groups.
%The construction of $V$ used in the first proof of Theorem \ref{T:FiniteUnram} can
%be expected to
%work also for general algebraic groups. Recall that we write $K$ as the field of
%rational functions on a regular integral scheme $X$ of finite type
%over $\mathrm{Spec} \: A$, where $A$ is either a finite field or the
%ring of $S$-integers in a number field for some finite set $S$ of
%its places. We take $V$ to be the set of discrete valuations of $K$
%associated with the divisors on $X$ (such a $V$ automatically satisfies
%conditions (A) and (B)). Any set $V$ of discrete valuations valuations of $K$
%obtained in this fashion for some choice of $X$ as above will be called
%\emph{divisorial}.

\vskip2mm

\noindent {\bf Conjecture 6.9.}
{\it Any divisorial set $V$ of discrete valuations of a
finitely generated field $K$ satisfies property $(\Phi)$.}

\addtocounter{thm}{1}

Over a number field $K$, the assertion of Conjecture 6.9 is
an easy consequence of the finiteness results for Galois cohomology ---
see \,\cite[Ch. III, 4.6]{Serre} since a semisimple group over a finite extension of $\Q_p$ that has good reduction is necessarily quasi-split (cf. \cite[Theorem 6.7]{Pl-R}). (Interestingly, there are nonsplit groups
over $\Q$ that have good reduction at \emph{all} primes, see \cite{Gross}, \cite{Conrad}, but there
are no abelian varieties over $\Q$ with smooth reduction everywhere \cite{Font}.) Furthermore, the finiteness of ${}_n\Br(K)_V$ implies Conjecture 6.9 for inner forms of type $\textsf{A}_{\ell}$. We also have the following conditional results for spinor groups.

Let $\mu_2 = \{ \pm 1\}$. Then for any discrete valuation $v$ of $K$
such that $\text{char}~\overline{K}_v \neq 2$ and any $i
\geqslant 1$, one can define a residue map in Galois cohomology
$$
\rho^i_v \colon H^i(K , \mu_2) \to H^{i-1}(\overline{K}_v , \mu_2)
$$
extending the map (\ref{E:LocalRes1}) introduced in \S\ref{S:Ramification}, to all dimensions (see, e.g.,
\cite[3.3]{CT} or \cite[6.8]{GiSz} for the details). Then for any
set $V$ of discrete valuations of $K$ one defines the unramified
part $H^i(K , \mu_2)_V$ to be $\bigcap_{v \in V} \ker
\rho^i_v$ (of course, $H^2(K , \mu_2)_V =
{}_2\Br(K)_V$).
\begin{thm}{\rm (\cite{CRR4})}
Let $K$ and $V$ be as in Conjecture 6.9. Assume that
for any finite set $V_0 \subset V$, the unramified cohomology groups
$H^i(K , \mu_2)_{V \setminus V_0}$ are finite for all $i \geq
1$. Then for any $n \geq 5$ and any finite subset $V_0 \subset V$, the set of
$K$-isomorphism classes of the spinor groups
${\rm Spin}_n(q)$ having good reduction at all $v \in V
\setminus V_0$, is finite.
\end{thm}

It is important to point out that the range of potential consequences of Conjecture 6.9 goes beyond the finiteness of the genus (Conjecture 6.6) and includes, for example, the Finiteness Conjecture for weakly commensurable Zariski-dense subgroups (see \cite[Conjecture 6.1]{R-ICM}) as well the finiteness of the Tate-Shafarevich set in certain situations. More precisely, let $G$ be an absolutely almost simple simply connected algebraic group over a field $K$ of good characteristic, and let $V$ be a divisorial set of discrete valuations of $K$. Denote by
%Among the possible applications of Conjecture 6.9, we would like to mention a finiteness statement forhe Tate-Shafarevich set. Let $K$ and $V$ be as in Conjecture 6.8, and let $G$ be an absolutely almost simple simply connected $K$-group. Denote by
$$
\text{{\brus SH}}(\overline{G}) \: :=  \: \ker \left( H^1(K
, \overline{G}) \longrightarrow \prod_{v \in V} H^1(K_v ,
\overline{G}) \right)
$$
the Tate-Shafarevich set
for the corresponding adjoint group $\overline{G}$. We can pick a
finite subset $V_0 \subset V$ so that $G$ has good reduction at
all $v \in V \setminus V_0$. Suppose $\xi \in \text{{\brus
SH}}(\overline{G})$ and let $G' = {}_{\xi}G$ be the corresponding
twisted group. By our assumption, $G' \simeq G$ over $K_v$ for all $v \in V$, and consequently
$G'$ has smooth reduction at all $v \in V \setminus
V_0$. Now, assuming Conjecture 6.9, we can conclude that
the groups
%Assuming Conjecture \ref{Con:smooth}, we would have that the
groups ${}_{\xi}G$ for $\xi \in \text{{\brus SH}}(\overline{G})$
form finitely many $K$-isomorphism classes; in other words, the
image of $\text{{\brus SH}}(\overline{G})$ under the canonical map
$H^1(K , \overline{G}) \stackrel{\lambda}{\longrightarrow} H^1(K ,
\mathrm{Aut} \: G)$ is finite. But since $\overline{G} \simeq
\mathrm{Int} \: G$ has finite index in $\mathrm{Aut} \: G$, the
map $\lambda$ has finite fibers, which would give the finiteness of
$\text{{\brus SH}}(\overline{G})$.

We note that Theorem \ref{T:Main1} can be used not only to investigate the finiteness of the genus, but also to prove that in some situations the genus is trivial. For example, we have the following.

\begin{thm}{\rm (\cite{CRR4})}\label{T:Main10}
Let $G$ be a simple group of type $\textsf{G}_2$ over a field of rational functions $K = k(x)$, where $k$ is a
global field of characteristic $\neq 2$. Then $\gen_K(G)$ consists
of a single element.
\end{thm}

We also note the following stability statement.

\begin{thm}{\rm (\cite{CRR4})}\label{T:Stab-const}
Let $G$ be an absolutely almost simple simply connected algebraic
group over a finitely generated field $k$ of characteristic zero.
Then for the field $k(x)$ of rational functions, every $G' \in \gen_{k(x)}(G  \otimes_k k(x))$ is of the
form  $H \otimes_k k(x)$ with $H \in
\gen_k(G)$.
%Let $G$ be an absolutely almost simple simply connected algebraic
%group over a finitely generated field $k$ of characteristic zero.
%Then for the field $k(x)$ of rational functions, the genus $\gen_{k(x)}(G  \otimes_k k(x))$ consists of the
%$k(x)$-isomorphism classes of $H \otimes_k k(x)$ with $H \in
%\gen_k(G)$.
\end{thm}

Combining this theorem with Theorem \ref{T:WC13}, we conclude that if $G$ is an absolutely almost
simple simply connected algebraic group of type different from
$\textsf{A}_{\ell}$ ($\ell > 1$), $\textsf{D}_{2\ell + 1}$ ($\ell >
1$) and $\textsf{E}_6$ over a number field $k$, then $\gen_{k(x)}(G
\otimes_k k(x))$ consists of a single element.

We began the article by mentioning the result of Amitsur \cite{Ami} on finite-dimensional central division algebras having the same splitting fields and explaining the additional features one encounters if one consider only finite-dimensional splitting fields or just maximal subfields; this eventually led to
our definition of the genus of a division algebra and, later, of a simple algebraic group. Now, we would like to conclude with a different notion of the genus, which is also based on the consideration of maximal \'etale subalgebras or maximal tori, but at the same time incorporates the availability of infinite-dimensional splitting fields, which was the key in Amitsur's theorem. This (more functorial) notion was proposed by A.S.~Merkurjev. One defines the {\it motivic genus} $\gen_m(G)$ of an absolutely almost simple algebraic $K$-group $G$ as the set of $K$-isomorphism classes of $K$-forms $G'$ of $G$ that have the same isomorphism classes of maximal tori not only over $K$ but also over any field extension $F/K$. Then Amitsur's theorem implies that for $G = \mathrm{SL}_{1 , D}$, the motivic genus is always finite, and reduces to one element for $D$ of exponent two.
%Finally, we would like to present a (more functorial) variant of the genus of an absolutely almost simple $K$-group $G$, which was suggested by A.S.~Merkurjev. We define the {\it motivic genus} $\gen_m (G)$ of $G$ as the set of $K$-isomorphism
%classes of $K$-forms $G'$ of $G$ that have the same
%isomorphism classes of maximal tori not only over $K$ but
%also over any field extension $F/K$. As we mentioned in \S \ref{S:Introduction}, a well-known result of Amitsur
%\cite{Ami} asserts that if $D$ and $D'$ are
%finite-dimensional central division $K$-algebras such that every
%field extension $F/K$ which splits $D$ also splits $D'$, then $[D']$
%lies in the cyclic subgroup $\langle [D] \rangle$ of $\Br(K)$
%generated by $[D]$; in particular, if $G = \mathrm{SL}_{1,D}$, then $\gen_m(G)$ is finite for any $D$ and reduces to a single
%element for $D$ of exponent two.
Furthermore, according to a result
of Izhboldin \cite{Izhb}, given nondegenerate quadratic forms $q$
and $q'$ of {\it odd} dimension $n$ over a~field $K$ of
characteristic $\neq 2$ the following condition

\vskip1mm

\noindent $(\star)$ {\it $q$ and $q'$ have the same Witt index over
any extension $F/K$,}

\vskip1mm

\noindent implies that $q$ and $q'$ are scalar multiples of each
other (this conclusion being false for even-dimensional forms). It
follows that for $G = \mathrm{Spin}_n(q)$, with $n$ odd, $\vert \gen_m(G) \vert =1.$ We note
that condition $(\star)$ is equivalent to the fact that the
motives of $q$ and $q'$ in the category of Chow motives are
isomorphic (Vishik \cite{Vish1}, and also Vishik \cite[Theorem
4.18]{Vish2}, Karpenko \cite{Karp}), which motivated the choice of terminology for this version of the genus. Other groups have not yet been investigated.
%so one can call the genus
%defined above the {\it motivic genus}.

\vskip3mm

\noindent {\small {\bf Acknowledgements.} The first-named author was supported by the Canada Research Chair Program and by an NSERC research grant. The second-named author was partially supported by NSF grant DMS-1301800, BSF grant 201049 and the Humboldt Foundation. The third-named author was supported by an NSF Postdoctoral Fellowship. Part of this article was written during the Joint IMU-AMS meeting in Israel (June 2014), and the warm reception of  Bar-Ilan and Tel Aviv Universities hosting the meeting is thankfully acknowledged.}

\vskip5mm

\bibliographystyle{amsplain}

\begin{thebibliography}{100}

\bibitem{ANT} {\it Algebraic Number Theory}, edited by
J.W.S.~Cassels and A.~Fr\"ohlich, 2nd edition, London Math. Soc.,
2010.

\bibitem{Ami} S.\,Amitsur, {\it Generic splitting fields of central
simple algebras}, Ann. math. {\bf 62}(1955), 8-43.

\bibitem{CGu} V.~Chernousov, V.~Guletskii, {\it 2-torsion of the
Brauer group of an elliptic curve,} Doc. Math. 2001, Extra volume,
85-120.

\bibitem{CRR1} V.I.~Chernousov, A.S.~Rapinchuk, I.A.~Rapinchuk, {\it On the genus
of a division algebra}, C. R. Acad. Sci. Paris, Ser. I 350 (2012),
807-812.

\bibitem{CRR2} V.I.~Chernousov, A.S.~Rapinchuk, I.A.~Rapinchuk, {\it
The genus of a division algebra and the unramified Brauer group},
Bull. Math. Sci.  {\bf 3} (2013), 211-240.


\bibitem{CRR3} V.I.~Chernousov, A.S.~Rapinchuk, I.A.~Rapinchuk, {\it Estimating the size of the genus of a division algebra}, in preparation.

\bibitem{CRR4} V.I.~Chernousov, A.S.~Rapinchuk, I.A.~Rapinchuk, {\it On algebraic groups having the same maximal tori}, in preparation.

\bibitem{CT} J.-L.~Colliot-Th\'el\`ene,  {\it Birational invariants, purity,
and the Gersten conjecture}, in {\it $K$-theory and Algebraic Geometry:
Connections with Quadratic Forms and Division Algebras,} Proc. Symp.
Pure Math. {\bf 58}, part 1, 1-64, AMS, 1995.

\bibitem{CT-Bour} J.-L.~Colliot-Th\'el\`ene, {\it Groupes de Chow
des z\'ero-cycles sur les vari\'et\'es $p$-adiques [d'apr\`es
S.~Saito, K.~Sato et al.]}, S\'eminaire Bourbaki 2009-2010, exp.
1012, Ast\'erisque {\bf 339}(2011), 1-30.

\bibitem{CT-S} J.-L.~Colliot-Th\'el\`ene, S.~Saito, {\it
Z\'ero-cycles sur les vari\'et\'es $p$-adiques et  groupe de
Brauer}, IMRN 1996, n$^{\circ}$ 4, 151-160.


\bibitem{Conrad} B.~Conrad,  {\it Non-split reductive groups over $\Z$},
Proceedings of the Summer School on Group Schemes, Luminy, 2011.

\bibitem{Conrad1} B.~Conrad, {\it Reductive Group Schemes}, Proceedings of the Summer School on Group Schemes, Luminy 2011.

\bibitem{Del} P.~Deligne, {\it Cohomologie \'etale}, SGA
$4\frac{1}{2}$, Lect. Notes Math. {\bf 569}, Springer, 1977.

\bibitem{SGA3} M.~Demazure, A.~Grothendieck, {\it Sch\'emas en groupes III}, Lecture Notes in Math {\bf 153}, Springer, New York (1970).

\bibitem{DF} R.K.~Dennis, B.~Farb, {\it Noncommutative algebra}, Springer, GTM 144, 1993.

\bibitem{Font} J.-M.~Fontaine,  {\it Il n'y a pas de vari\'et\'e
ab\'elienne sur $\mathbb{Z}$}, { Invent. math.} {\bf 81}(1985),
515-538.

\bibitem{For} O.~Forster, {\it Lectures on Riemmann Surfaces}, Springer, GTM 81, 1981.

\bibitem{Fuj} K.~Fujiwara, {\it A proof of the absolute purity
conjecture (after Gabber)}, Algebraic Geometry 2000, Azumino
(Hotaka), 153-183, Adv. Stud. Pure Math., {\bf 36}, Math. Soc.
Japan, 2002.

\bibitem{GarR} S.~Garibaldi,  A.S.~Rapinchuk,  {\it Weakly
commensurable $S$-arithmetic subgroups in almost simple algebraic
groups of types $\textsf{B}$ and $\textsf{C}$}, {Algebra and
Number Theory} {\bf 7}(2013), no. 5, 1147-1178.

\bibitem{GS} S.~Garibaldi, D.~Saltman, {\it Quaternion Algebras
with the Same Subfields}, Quadratic forms, linear algebraic groups,
and cohomology,  225-238, Dev. math. 18, Springer, New York, 2010.


\bibitem{GiSz} P.\,Gille, T.\,Szamuely, {\it Central Simple Algebras
and Galois Cohomology}, Cambridge Univ. Press, 2006.

\bibitem{Gross} B.H.~Gross,  Groups over $\mathbb{Z}$, {\it Invent.
math.} {\bf 124}(1996), 263-279.

\bibitem{Izhb} O.T.~Izhboldin, {\it Motivic Equivalence of Quadratic
Forms}, Doc. Math. {\bf 3}(1998), 341-351.

\bibitem{Kac} M. Kac, {\it Can one hear the shape of a drum?}, {Amer.
Math. Monthly} {\bf 73}(1966), no. 4, part 2, 1-23.

\bibitem{Karp} N.~Karpenko, {\it Criteria of motivic equivalence for
quadratic forms and central simple agebras}, Math. Ann. {\bf
317}(2000), no. 3, 585-611.

\bibitem{Common} D.~Krashen, E.~Matzri, A.S.~Rapinchuk, L.~Rowen, D.~Saltman, {\it Division algebras with common subfields}, in preparation.

\bibitem{KM} D.~Krashen, K.~McKinnie, {\it Distinguishing algebras by their finite splitting fields}, Manuscr. math. {\bf 134} (2011), no. 1-2, 171-182.

\bibitem{LaDG} S.~Lang, {\it Fundamentals of Diophantine Geometry}, Springer, 1983.

\bibitem{Lich} S.~Lichtenbaum, {\it Duality Theorems for Curves over $p$-adic Fields}, Invent. Math. {\bf 7}, 120-136 (1969).

\bibitem{LMPT} B.~Linowitz, D.B.~McReynolds, P.~Pollack, L.~Thompson, {\it Counting and effective rigidity in algebra and geometry}, arxiv1407.2294

\bibitem{MR} C.~Maclachlan, A.~Reid, {\it The Arithmetic of Hyperbolic 3-Manifolds}, Springer, GTM 219, 2003.

\bibitem{Meyer} J.S.~Meyer, {\it A division algebra with infinite
genus}, Bull. London Math. Soc. {\bf 46}(2014), 463-468.  

\bibitem{Mil} J.~Milne, {\it Class Field Theory}, available at www.jmilne.org

\bibitem{Pierce} R.S.~Pierce, {\it Associative Algebras}, GTM 88, Springer, 1982.

\bibitem{Pl-R} V.P.~Platonov, A.S.~Rapinchuk, {\it Algebraic Groups
and Number Theory}, Academic Press, 1994.

\bibitem{PR1} G.~Prasad, A.S.~Rapinchuk, {\it Weakly
commensurable arithmetic groups and isospectral locally symmetric
spaces,} Publ. math. IHES {\bf 109}(2009), 113-184.

\bibitem{PR-Fields} G.~Prasad, A.S.~Rapinchuk,  {\it On the fields
generated by the lengths of closed geodesics in locally symmetric
spaces}, Geom. Dedicata {\bf 172}(2014), 79-120. 

\bibitem{PR-Gen} G.~Prasad, A.S.~Rapinchuk, {\it Generic elements in
Zariski-dense subgroups and isospectral locally symmetric spaces}, in
{\it Thin Groups and Superstrong Approximation}, 211-252. MSRI
Publications, {\bf 61}(2014).

\bibitem{R-ICM} A.S.~Rapinchuk, {\it Towards the eigenvalue rigidity of Zariski-dense subgroups},  Proceedings of the ICM (2014), Seoul, 
247-269. 

\bibitem{RR} A.S.~Rapinchuk, I.A.~Rapinchuk, {\it On division
algebras having the same maximal subfields,} Manuscr. math. {\bf
132}(2010), 273-293.

\bibitem{Reid} A. Reid, {\it Isospectrality and commensurability of
arithmetic hyprebolic 2- and 3-manifolds}, {Duke Math. J.} {\bf
65}(1992), 215-228.

\bibitem{Roq} P.~Roquette, {\it The Brauer-Hasse-Noether Theorem in
Historical Perspective}, Springer, 2005.

\bibitem{Salt} D.~Saltman, {\it Lectures on division
algebras}, CBMS Regional Conference Series, vol. 94, AMS, 1999.

\bibitem{Serre} Serre, J.-P.,  {\it Galois Cohomology.} Springer,
1997.

\bibitem{Se} J.-P.~Serre, {\it Local Fields,} GTM 67, Springer,
1979.

\bibitem{Silv} J.H.~Silverman, {\it The Arithmetic of Elliptic Curves}, 2nd edition, Springer, GTM 106, 2009.

\bibitem{Tikh} S.V.~Tikhonov, {\it Division algebras of prime degree with infinite genus}, arXiv1407.5041.

\bibitem{Vin2} E.B.~Vinberg,  {\it Some examples of Fuchsian groups
sitting in $SL_2(\mathbb{Q})$}, preprint 12011 of the SFB-701.
Universit\"at Bielefeld (2012).

\bibitem{Vish1} A.~Vishik, {\it Integral motives of quadrics}, Max
Planck Institute f\"ur Mathematik, Bonn, preprint MPI-1998-13, 1-82.


\bibitem{Vish2} A.~Vishik, {\it Motives of quadrics with
applications to the theory of quadratic forms}, {\it Geometric
methods in the algebraic theory of quadratic forms}, 25-101, LNM
{\bf 1835}, Springer, 2004.

\bibitem{W2} A.R.~Wadsworth, {\it Valuation theory on finite
dimensional division algebras}, In: Valuation Theory and Its
Applications, vol. I. Saskatoon, SK, (1999), Fields Inst. Commun.
vol. 32, pp. 385-449, Amer. Math. Soc., Providence, RI, (2002).


\end{thebibliography}

\end{document}